\documentclass[12pt,reqno]{amsart} 
\usepackage{amssymb}
\usepackage{mathrsfs}
\usepackage{enumerate}
\usepackage[all]{xy}
\usepackage[usenames]{color}
\usepackage[colorlinks=true, citecolor=blue, linkcolor=blue]{hyperref}

\renewenvironment{enumerate}[1][]
{\begin{enumerat}[#1]\setlength{\itemsep}{6pt}}{\end{enumerat}}

\newenvironment{enuma}{\begin{enumerate}[{\rm(a) }]}{\end{enumerate}}

\renewenvironment{itemize}
{\begin{itemiz}\setlength{\itemsep}{6pt} \setlength{\topsep}{28pt}
\setlength{\itemindent}{-20pt}\setlength{\leftmargin}{50pt}%
}{\end{itemiz}}

\definecolor{darkgreen}{RGB}{0,120,0}
\definecolor{darkred}{RGB}{200,0,0}
\definecolor{bluegreen}{RGB}{0,100,100}
\definecolor{newestcolor}{RGB}{200,100,0}

\newcommand{\fr}[1]{}

\numberwithin{table}{section}

\newlength{\short}
\newcommand{\4}[1]{\protect\widebar{#1}}
\newcommand{\5}[1]{\widehat{#1}}

\let\8=\scriptstyle
\newcommand{\9}[1]{{}^{#1}\!}

\newcommand{\bb}{\mathfrak{b}}

\newcommand{\qq}{\mathfrak{q}}

\def\pair[#1,#2]{[\hskip-1.5pt[#1,#2]\hskip-1.5pt]}

\SelectTips{cm}{10} \UseTips   

\let\oldcirc=\circ
\renewcommand{\circ}{\mathchoice
    {\mathbin{\scriptstyle\oldcirc}}{\mathbin{\scriptstyle\oldcirc}}
    {\mathbin{\scriptscriptstyle\oldcirc}}
    {\mathbin{\scriptscriptstyle\oldcirc}}}

\def\beq#1\eeq{\begin{equation*}#1\end{equation*}}
\def\beqq#1\eeqq{\begin{equation}#1\end{equation}}

\numberwithin{equation}{section}

\newtheorem{Thm}{Theorem}[section]
\newtheorem{Prop}[Thm]{Proposition}

\newtheorem{Lem}[Thm]{Lemma}
\newtheorem{Defi}[Thm]{Definition}

\newtheorem{Ex}[Thm]{Example}

\newtheorem{Th}{Theorem}

\newcommand{\widebar}[1]
      {\overset{{\mskip1mu\leaders\hrule height0.4pt\hfill\mskip1mu}}{#1}
      \vphantom{#1}}

\newcounter{let} 
\loop\stepcounter{let}
\expandafter\edef\csname cal\alph{let}\endcsname%
{\noexpand\mathcal{\Alph{let}}}
\ifnum\thelet<26\repeat

\loop\stepcounter{let}
\expandafter\edef\csname scr\alph{let}\endcsname%
{\noexpand\mathscr{\Alph{let}}}
\ifnum\thelet<26\repeat

\newcommand{\tdef}[2][]{\expandafter\newcommand\csname#2\endcsname%
{#1\textup{#2}}}
\tdef{Iso}   \tdef{Aut}    \tdef{Out}    \tdef{Inn}    \tdef{Hom}
\tdef{End}   \tdef{Inj}    \tdef{map}    \tdef{Ker}    \tdef{Ob}
\tdef{Mor}   \tdef{Res}    \tdef{Id}     \tdef{Fr}     \tdef{Spin} 
\tdef{rk}    \tdef{conj}   \tdef{incl}   \tdef{proj}   \tdef{diag} 
\tdef{trf}   \tdef{Sol}    \tdef{Sz}     \tdef{cj}
\tdef{Rep}   \tdef{pr}    \tdef{Inndiag} \tdef{Outdiag}  \tdef{expt}
\tdef{supp}  \tdef{Isom}   \tdef{ord}    \tdef{Coker}   \tdef{Tr}
\tdef[_]{typ} \tdef[^]{op} \tdef[^]{ab}   \tdef{lcm}

\newcommand{\fdef}[1]{\expandafter\newcommand\csname#1\endcsname%
{\mathfrak{#1}}}
\fdef{X}  \fdef{red}  \fdef{foc}  \fdef{hyp}  \fdef{Lie} \fdef{Y} 
\fdef{Perm}

\newcommand{\bbdef}[1]{\expandafter\newcommand%
\csname#1\endcsname{\mathbb{#1}}}
\bbdef{C} \bbdef{F} \bbdef{R} \bbdef{Z} \bbdef{N} \bbdef{Q} \bbdef{K}

\newcommand{\itdef}[1]{\expandafter\newcommand\csname#1\endcsname%
{\textit{#1}}}
\itdef{PSL}  \itdef{PSU}  \itdef{SL}  \itdef{SU}  \itdef{GL} \itdef{GU}
\itdef{Sp}   \itdef{PSp} \itdef{PSO} \itdef{SO}   \itdef{SD} \itdef{PGU} 
\itdef{PGL}  \itdef{Co}  \itdef{Fi}  \itdef{GO}   \itdef{PGO} \itdef{HS}
\itdef{He}   \itdef{McL} \itdef{Suz}  \itdef{Ly}  \itdef{Ru}
\newcommand{\ON}{\textit{O'N}}

\newcommand{\POmega}{\textit{P}\varOmega}

\newcommand{\pSL}{\textit{(P)SL}}
\newcommand{\pSU}{\textit{(P)SU}}

\newcommand{\gee}{\varepsilon}
\newcommand{\sminus}{\smallsetminus}
\newcommand{\lie}[3]{\def\test{#2}\def\tst{G}\ifx\test\tst{{}^{#1}#2_{#3}}
\else{{}^{#1}\!#2_{#3}}\fi}
\renewcommand{\*}{\,\lower6pt\hbox{\Large{\textup{*}}}\,}
\newcommand{\syl}[2]{\textup{Syl}_{#1}(#2)}
\newcommand{\sylp}[1]{\syl{p}{#1}}

\renewcommand{\Im}{\textup{Im}}
\newcommand{\autf}{\Aut_{\calf}}

\newcommand{\outf}{\Out_{\calf}}

\newcommand{\homf}{\Hom_{\calf}}
\newcommand{\isof}{\Iso_{\calf}}
\newcommand{\defeq}{\overset{\textup{def}}{=}}

\newcommand{\mxfoura}[8]{\left(\begin{smallmatrix}#1&#2&#3&#4\\#5&#6&#7&#8}
\newcommand{\mxfourb}[8]{\\#1&#2&#3&#4\\#5&#6&#7&#8\end{smallmatrix}\right)}

\renewcommand{\:}{\colon}

\newcommand{\nsg}{\trianglelefteq}
\let\nnsg=\ntrianglelefteq

\newcommand{\til}[1]{\widetilde{#1}}
\renewcommand{\gg}{\mathbb{G}}

\newcommand{\gen}[1]{{\langle}#1{\rangle}}
\newcommand{\Gen}[1]{{\bigl\langle}#1{\bigr\rangle}}

\newcommand{\longleft}[1]{\;{\leftarrow%
\count255=0 \loop \mathrel{\mkern-6mu}%
    \relbar\advance\count255 by1\ifnum\count255<#1\repeat}\;}
\newcommand{\longright}[1]{\;{\count255=0 \loop \relbar\mathrel{\mkern-6mu}%
    \advance\count255 by1\ifnum\count255<#1\repeat\rightarrow}\;}
\newcommand{\Right}[2]{\overset{#2}{\longright#1}}
\newcommand{\RIGHT}[3]{\mathrel{\mathop{\kern0pt\longright#1}
        \limits^{#2}_{#3}}}

\newcommand{\LEFT}[3]{\mathrel{\mathop{\kern0pt\longleft#1}\limits^{#2}_{#3}}
}
\newcommand{\dRIGHT}[3]{\mathrel{%
   \mathop{\vcenter{\baselineskip=0pt\hbox{$\kern0pt\longright#1$}%
   \hbox{$\kern0pt\longright#1$}}}\limits^{#2}_{#3}}}
\newcommand{\LRIGHT}[3]{\mathrel{%
   \mathop{\vcenter{\baselineskip=0pt\hbox{$\kern0pt\longleft#1$}%
   \hbox{$\kern0pt\longright#1$}}}\limits^{#2}_{#3}}}
\newcommand{\RLEFT}[3]{\mathrel{%
   \mathop{\vcenter{\baselineskip=0pt\hbox{$\kern0pt\longright#1$}%
   \hbox{$\kern0pt\longleft#1$}}}\limits^{#2}_{#3}}}
\newcommand{\onto}[1]{\;{\count255=0 \loop \relbar\mathrel{\mkern-6mu}%
    \advance\count255 by1
    \ifnum\count255<#1 \repeat \twoheadrightarrow}\;}
\newcommand{\Onto}[2]{\overset{#2}{\onto#1}}

\let\too=\longrightarrow
\let\xto=\xrightarrow

\newcommand{\quotfus}[2][p]{\Gamma_{#1'}(#2)}

\newcommand{\thet}[2]{\theta_{#1}^{(#2)}}

\begin{document}


\title{Simplicity of fusion systems of finite simple groups}

\author{Bob Oliver}
\address{Universit\'e Sorbonne Paris Nord, LAGA, UMR 7539 du CNRS, 
99, Av. J.-B. Cl\'ement, 93430 Villetaneuse, France.}
\email{bobol@math.univ-paris13.fr}
\thanks{B. Oliver is partially supported by UMR 7539 of the CNRS}

\author{Albert Ruiz}
\address{Departament de Matem\`atiques, Edifici Cc, Universitat Aut\`onoma de 
Barcelona, 08193 Cerdanyola del Vall\`es (Barcelona), Spain.
\newline\indent Centre de 
Recerca Matem\`atica, Edifici Cc, Campus de Bellaterra, 08193 Cerdanyola del 
Vall\`es (Barcelona), Spain.}
\email{albert@mat.uab.cat}
\thanks{A. Ruiz is partially supported by MICINN grant MTM2016-80439-P and
AGAUR grant 2017-SGR-1725}

\date{}


\begin{abstract}
We determine for which known finite simple groups $G$ and which primes $p$ 
the $p$-fusion system of $G$ is simple. This means first collecting 
together the results that were already known (and correcting two errors made 
in an earlier study of this question), and then handling the remaining 
cases. At the same time, we develop some new tools to use when determining 
$O^{p'}(\calf)$ for arbitrary saturated fusion systems.
\end{abstract}

\maketitle

For a prime $p$ and a finite group $G$, the \emph{fusion system} of $G$ 
over a Sylow $p$-subgroup $S$ of $G$ is the category $\calf_S(G)$ whose 
objects are the subgroups of $S$, and whose morphisms are those 
homomorphisms between subgroups induced by conjugation by elements of $G$. 
The fusion system of $G$ thus encodes its local structure: the conjugacy 
relations among its $p$-subgroups and $p$-elements. Motivated by 
connections with modular representation theory, Puig in the 1990s defined 
the concept of abstract fusion systems (published later in \cite{Puig}): an 
(abstract) fusion system $\calf$ over a finite $p$-group $S$ is a category 
whose objects are the subgroups of $S$, and whose morphisms are injective 
homomorphisms satisfying certain conditions motivated by properties of 
finite groups such as the Sylow theorems. (See Section \ref{s:defs} for 
more detail.)

Normal fusion subsystems and simple fusion systems are defined by analogy 
to those of finite groups. Our main result here is to determine for exactly 
which known finite simple groups $G$ and which primes $p$ the $p$-fusion 
system of $G$ is simple. This question was studied and answered in most 
cases by Aschbacher in Chapter 16 of \cite{A-gfit}, but a few cases (all 
involving groups of Lie type in cross characteristic) were left open, and 
it is on those that we focus here. We also correct two errors found among 
Aschbacher's conclusions (explained in the proof of Proposition 
\ref{t:tame.simple}(d)). One problem of particular interest is that of for 
which $G$ and $p$, the $p$-fusion system of $G$ contains a normal subsystem 
of index prime to $p$ that is exotic, and we are also able to describe this 
situation quite precisely. 


Our detailed group-by-group results are stated in Propositions 
\ref{t:tame.simple} (alternating and sporadic groups and groups of Lie type 
in characteristic $p$), \ref{p:2'-index} (groups of Lie type in odd 
characteristic when $p=2$), \ref{p:p'-ind-1} (classical groups in 
characteristic different from $p$ when $p$ is odd), and \ref{p:p'-ind-2} 
(exceptional groups of Lie type in characteristic different from $p$ for 
$p$ odd). They are then summarized by the following theorem, which will be 
restated with more details as Theorem \ref{F.not.simple}. 

\begin{Th} \label{ThA}
Fix a prime $p$ and a known finite simple group $G$ such that $p\mid|G|$. 
Fix $S\in\sylp{G}$, and set $\calf=\calf_S(G)$. Then either 
\begin{enuma} 

\item $S\nsg\calf$; or 

\item $p=3$ and $G\cong G_2(q)$ for some $q\equiv\pm1$ (mod $9$), in which 
case $|O_3(\calf)|=3$, and $O^{3'}(\calf)<\calf$ is realized by 
$\SL_3^\pm(q)$; or 

\item $p\ge5$, $G$ is one of the classical groups $\PSL_n^\pm(q)$, 
$\PSp_{2n}(q)$, $\varOmega_{2n+1}(q)$, or $\POmega_{2n+2}^\pm(q)$ where $n\ge2$ 
and $q\not\equiv0,\pm1$ (mod $p$), in which case $O^{p'}(\calf)$ is simple 
and exotic; or 

\item $O^{p'}(\calf)$ is simple, and it is realized by a known finite 
simple group $G^*$. 

\end{enuma}
Moreover, in case \textup{(c)}, there is a subsystem $\calf_0\nsg\calf$ 
of index at most $2$ in $\calf$ with the property that for 
each saturated fusion system $\cale$ over $S$ such that 
$O^{p'}(\cale)=O^{p'}(\calf)$, $\cale$ is realizable if and only 
if it contains $\calf_0$. 
\end{Th}

The statements in Theorem \ref{ThA} that certain fusion systems are exotic 
are proved using the classification of finite simple groups. The 
theorem has been formulated and proven so that the other statements 
are independent of this.

Here, by analogy with the notation used for finite groups, $O^{p'}(\calf)$ 
denotes the smallest (saturated) fusion subsystem of index prime to $p$ 
(see Definition \ref{d:p'-index} and Proposition \ref{p:thetaF}(b)). Also, 
a saturated fusion system is called \emph{realizable} if it is 
isomorphic to the fusion system of some finite group, and is called 
\emph{exotic} otherwise.

Among the tools developed to determine $O^{p'}(\calf)$ in these cases, 
we note in particular the following, which reduce the information 
needed in two different directions: 
\begin{itemize}

\item Although $O^{p'}(\calf)$ is characterized by its restriction to 
$\calf$-centric subgroups, we can, in fact, get the information we need 
from the smaller family of subgroups that are $\calf$-centric and 
$\calf$-radical, as described in Proposition \ref{p:pi1(Fcr)}.

\item Although subsystems of index prime to $p$ are determined by their 
automizers on $S$, it is sometimes more convenient to deal with the 
automizer of some subgroup $A\nsg S$ that is $\calf$-centric and weakly 
closed in $\calf$ (such as the $p$-power torsion in a maximal torus in a 
group of Lie type). Tools for doing this are developed in Section 
\ref{s:w.cl.}.

\end{itemize}

This paper was originally motivated by our work (still continuing) with 
Carles Broto and Jesper M\o{}ller on tameness of realizable fusion systems 
(see \cite[Definition III.6.3]{AKO}). While trying to determine whether or 
not all realizable fusion systems are tame, we found that it is first 
necessary to understand more precisely the normal fusion subsystems of 
fusion systems of simple groups. Independently of that, some of our results 
are used in recent work of Radha Kessar, Gunter Malle, and Jason 
Semeraro to calculate weights (in the context of the Alperin weight 
conjecture in modular representation theory) attached to exotic fusions 
arising from homotopy fixed points of $p$-compact groups.

We begin the paper with a brief introduction to fusion systems in Section 
\ref{s:defs}, and an introduction to (normal) subsystems of index prime to 
$p$ in Section \ref{s:Op'F}. We then develop tools, in the last part of 
Section \ref{s:Op'F} and in Section \ref{s:w.cl.}, to help determine 
$O^{p'}(\calf)$ for a given fusion system $\calf$. Finally, in Section 
\ref{s:simplicity}, we go through the list of all fusion systems of known 
finite simple groups at primes dividing their order, starting with a survey 
of the cases already handled by Aschbacher in \cite[Chapter 16]{A-gfit}, to 
determine which of them are simple. 

The authors would like to thank very much the referee for reading very 
carefully through the paper and making many helpful suggestions.

\textbf{Notation: } As usual, for a finite group $G$, $O_p(G)$ and 
$O_{p'}(G)$ denote the largest normal subgroups of $p$-power order and 
of order prime to $p$, respectively, while $O^p(G)$ and $O^{p'}(G)$ 
are the smallest normal subgroups of $p$-power index and of index 
prime to $p$. We follow the standard conventions by setting 
$\pSL_n^+(q)=\pSL_n(q)$ and $\pSL_n^-(q)=\pSU_n(q)$, as well as 
$E_6^+(q)=E_6(q)$ and $E_6^-(q)=\lie2E6(q)$.


\section{Definitions and earlier results}
\label{s:defs}

We begin by listing some of the basic definitions and properties of fusion 
systems. We use \cite[Part I]{AKO} as our main reference, but most or all 
of the following definitions are due originally to Puig: first in 
unpublished notes and then in \cite{Puig}.

For a prime $p$, a \emph{fusion system} over a finite $p$-group $S$ is a 
category whose objects are the subgroups of $S$, and whose morphisms are 
injective homomorphisms between subgroups such that for each $P,Q\le S$:
\begin{itemize} 
\item $\homf(P,Q)\supseteq \Hom_S(P,Q) \defeq \{ c_g=(x\mapsto\9gx) 
\,|\, g\in S, ~ \9gP\le Q \}$; and 
\item for each $\varphi\in\homf(P,Q)$, 
$\varphi^{-1}\in\homf(\varphi(P),P)$.
\end{itemize}
Thus $\homf(P,Q)$ denotes the set of morphisms in $\calf$ from $P$ to 
$Q$. We also write $\isof(P,Q)$ for the set of isomorphisms, 
$\autf(P)=\isof(P,P)$, and $\outf(P)=\autf(P)/\Inn(P)$. For 
$P\le S$ and $g\in S$, we set 
	\[ P^\calf = \{\varphi(P)\,|\,\varphi\in\homf(P,S)\} 
	\qquad\textup{and}\qquad 
	g^\calf = \{\varphi(g)\,|\,\varphi\in\homf(\gen{g},S)\}  \]
(the sets of subgroups and elements \emph{$\calf$-conjugate} to $P$ and 
to $g$). 

Since we will need to refer to the Sylow and extension axioms on several 
occasions, the following version of the definition of a saturated fusion 
system seems the most convenient one to give here. (See Definitions I.2.2 
and I.2.4 and Proposition I.2.5 in \cite{AKO}.)

\begin{Defi} \label{d:sfs}
Let $\calf$ be a fusion system over a finite $p$-group $S$.
\begin{enuma} 

\item A subgroup $P\le S$ is \emph{fully normalized} \emph{(fully 
centralized)} in $\calf$ if $|N_S(P)|\ge|N_S(Q)|$ ($|C_S(P)|\ge|C_S(Q)|$) 
for each $Q\in P^\calf$. 

\item The fusion system $\calf$ is \emph{saturated} if it satisfies the 
following two conditions:
\begin{enumerate}[\rm(I) ]

\item \emph{(Sylow axiom)} For each subgroup $P\le S$ fully normalized in 
$\calf$, $P$ is fully centralized and $\Aut_S(P)\in\sylp{\autf(P)}$.

\item \emph{(extension axiom)} For each isomorphism $\varphi\in\isof(P,Q)$ 
in $\calf$ such that $Q$ is fully centralized in $\calf$, $\varphi$ extends 
to a morphism $\4\varphi\in\homf(N_\varphi,S)$ where 
	\[ N_\varphi = \{ g\in N_S(P) \,|\, \varphi c_g \varphi^{-1} \in 
	\Aut_S(Q) \}. \]

\end{enumerate}
\end{enuma}
\end{Defi}

If $G$ is a finite group and $P,Q\le G$, then 
$\Hom_G(P,Q)\subseteq\Hom(P,Q)$ denotes the set of (injective) 
homomorphisms induced by conjugation in $G$. If $S\in\sylp{G}$, then 
$\calf_S(G)$ denotes the fusion system over $S$ where 
$\Hom_{\calf_S(G)}(P,Q)=\Hom_G(P,Q)$ for $P,Q\le S$, and $\calf_S(G)$ is 
saturated by \cite[Theorem I.2.3]{AKO}. A saturated fusion system $\calf$ 
over $S$ will be called \emph{realizable} if $\calf=\calf_S(G)$ for some 
finite group $G$ with $S\in\sylp{G}$, and will be called \emph{exotic} 
otherwise.

We next list some of the terminology used to describe certain 
subgroups in a fusion system.

\begin{Defi} \label{d:subgroups}
Let $\calf$ be a fusion system over a finite $p$-group $S$, and let $P\le 
S$ be a subgroup.
\begin{enuma} 

\item $P$ is \emph{$\calf$-centric} if $C_S(Q)\le Q$ for each $Q\in 
P^\calf$.

\item $P$ is \emph{$\calf$-radical} if $O_p(\outf(P))=1$. 

\item $\calf^c\supseteq\calf^{rc}$ denote the sets of $\calf$-centric, and 
$\calf$-centric $\calf$-radical, subgroups of $S$, and also (depending on 
context) the full subcategories of $\calf$ with those objects.

\item $P$ is \emph{weakly closed} in $\calf$ if $P^\calf=\{P\}$.

\item $P$ is \emph{strongly closed} in $\calf$ if for each $x\in P$, 
$x^\calf\subseteq P$.

\item $P$ is \emph{normal} in $\calf$ ($P\nsg\calf$) if each 
$\varphi\in\homf(Q,R)$ (for $Q,R\le S$) extends to a morphism 
$\4\varphi\in\homf(PQ,PR)$ such that $\4\varphi(P)=P$.

\item $P$ is \emph{central} in $\calf$ if each $\varphi\in\homf(Q,R)$ (for 
$Q,R\le S$) extends to a morphism $\4\varphi\in\homf(PQ,PR)$ such that 
$\4\varphi|_P=\Id_P$.

\item $O_p(\calf)\ge Z(\calf)$ denote the (unique) largest normal and 
central subgroups, respectively, in $\calf$.

\end{enuma}
\end{Defi}

When $X$ is a set of homomorphisms between subgroups of a given finite 
$p$-group $S$, we let $\gen{X}$ denote the smallest fusion system over $S$ 
(not necessarily saturated) that contains $X$ among its morphisms. In other 
words, $\gen{X}$ is the category whose objects are the subgroups of $S$, 
and whose morphisms are composites of restrictions of homomorphisms in 
$X\cup\Inn(S)$. Similarly, when $\cale$ is a fusion system over $S$ or a 
subgroup of $S$, and $X$ is a set of homomorphisms between subgroups of 
$S$, we set $\gen{\cale,X}=\gen{\Mor(\cale)\cup X}$: the smallest 
fusion system over $S$ containing $\cale$ and $X$.

The following version of Alperin's fusion theorem for fusion systems will 
suffice for our purposes here. We refer to Theorem I.3.5 and Proposition 
I.3.3(a) in \cite{AKO}.

\begin{Thm} \label{AFT} 
For each saturated fusion system $\calf$ over a finite $p$-group $S$, 
	\[ \calf = \Gen{ \autf(P) \,\big|\, \textup{$P\in\calf^{rc}$ and is 
	fully normalized in $\calf$} }. \]
Equivalently, each morphism in $\calf$ is a composite of restrictions of 
$\calf$-automorphisms of subgroups that are fully normalized, 
centric, and radical in $\calf$.
\end{Thm}

We finish the section with a much more specialized lemma: one which will be 
useful in the last section when showing the nonexistence of certain fusion 
systems. 

\begin{Lem} \label{p:p-1} \fr4
Let $p$ be an odd prime, and let $\calf$ be a saturated fusion system over 
a finite $p$-group $S$. Assume that there is an abelian subgroup $A\nsg S$ 
such that $|S/A|=p$, $|[S,A]|\ge p^2$, and $A\nnsg\calf$. Set $G=\autf(A)$ 
and $U=\Aut_S(A)\in\sylp{G}$; then $|\Aut_G(U)|=p-1$.
\end{Lem}

\begin{proof} Since $|[S,A]|>p$, $A$ is the unique abelian subgroup of 
index $p$ in $S$. By Lemmas 2.2(a,c) and 2.3(a,b) in \cite{indp2} and 
since $A\nnsg\calf$, there is a subgroup $A\ne P<S$ in one of the classes 
$\calh$ or $\calb$ of subgroups defined in \cite[Notation 2.1]{indp2}, 
maximal among all $\calf$-centric $\calf$-radical subgroups properly 
contained in $S$, and with the properties that $|N_S(P)/P|=p$ and 
$O^{p'}(\outf(P))\cong\SL_2(p)$. Choose $\alpha\in O^{p'}(\autf(P))$ of 
order prime to $p$ whose class in $O^{p'}(\outf(P))$ normalizes 
$\Out_S(P)\cong C_p$ and has order $p-1$. Then $\alpha$ extends to an 
element of $\autf(N_S(P))$ by the extension axiom (Definition 
\ref{d:sfs}(b)), and hence by the maximality of $P$ and Theorem \ref{AFT} 
(Alperin's fusion theorem) to some $\4\alpha\in\autf(S)$. Then 
$\4\alpha|_A$ normalizes $U$, and its class in $\Aut_G(U)$ has order $p-1$. 
\end{proof}


\section{Subsystems of index prime to \texorpdfstring{$p$}{p}}
\label{s:Op'F}

Subsystems of fusion systems of index prime to $p$ were first studied by 
Puig \cite[\S\,6.5]{Puig}, and later in \cite[\S\,3,5]{BCGLO2}. 
We begin the section with some basic results, where for simplicity we give 
references only in \cite{AKO} when possible. After covering the basic 
properties, we develop some tools which will be used in the next section to 
determine $O^{p'}(\calf)$ and its index in $\calf$ in specific cases.

Note in the following definition that the analogy for a group $G$ is 
``subgroups of $G$ that contain $O^{p'}(G)$'', not arbitrary subgroups of 
index prime to $p$ in $G$. For example, the fusion system of $S$ is not, in 
general, of index prime to $p$ in an arbitrary fusion system $\calf$ over 
$S$.

\begin{Defi} \label{d:p'-index} \fr2
Let $\calf$ be a saturated fusion system over a finite $p$-group $S$. 
\begin{enuma} 

\item \cite[Theorem I.7.7.d]{AKO} Let $O^{p'}_*(\calf)\le\calf$ be 
the fusion subsystem (not necessarily saturated) generated by the groups 
$O^{p'}(\autf(P))$ for all $P\le S$, and set 
	\begin{align*} 
	\autf^0(S) &= \Gen{\alpha\in\autf(S) \,\big|\, 
	\alpha|_P\in \Hom_{O^{p'}_*(\calf)}(P,S), 
	~\textup{some $P\in\calf^c$} } \\
	\quotfus\calf &= \autf(S)/\autf^0(S). 
	\end{align*}

\item A fusion subsystem $\cale\le\calf$ has \emph{index prime to $p$} in 
$\calf$ if $\cale$ contains $O^{p'}_*(\calf)$; equivalently, if $\cale$ is 
a fusion system over $S$ and $\Aut_\cale(P)\ge O^{p'}(\autf(P))$ for each 
$P\le S$. For such a subsystem $\cale\le\calf$, the \emph{index} of 
$\cale$ in $\calf$ is defined to be the index of $\Aut_\cale(S)$ in 
$\autf(S)$.

\end{enuma}
\end{Defi}

Note that $O^{p'}_*(\calf)$ is denoted $\cale_0$ in \cite[Theorem 
I.7.7(d)]{AKO}.

\begin{Lem} \label{l:Ec=Fc}
Let $\calf$ be a saturated fusion system over a finite $p$-group $S$, and 
let $\cale\le\calf$ be a fusion subsystem (not necessarily saturated) 
of index prime to $p$. Then 
\begin{enuma} 
\item $\cale^c=\calf^c$ and $\cale^{cr}=\calf^{cr}$; 
\item a subgroup of $S$ is fully normalized (fully centralized) in $\cale$ 
if and only if it is fully normalized (fully centralized) in $\calf$; and 

\item \emph{(Frattini condition)} for each $P,Q\le S$ and 
$\varphi\in\homf(P,Q)$, there are $\alpha\in\autf(S)$ and 
$\varphi_0\in\Hom_\cale(P,\alpha_1^{-1}(Q))$ 
such that $\varphi=(\alpha|_{\alpha_1^{-1}(Q)})\circ\varphi_0$. 

\end{enuma}
\end{Lem}

\begin{proof} Point (c), and the equality $\cale^c=\calf^c$, are shown in 
\cite[Lemma I.7.6(a)]{AKO}. For all $P\le S$, we have 
$O^{p'}(\autf(P)) \le \Aut_\cale(P) \le \autf(P)$, and hence 
	\[ O_p(\autf(P))\le O_p(\Aut_\cale(P))\le O_p(O^{p'}(\autf(P)))
	\le O_p(\autf(P)), \]
where the last inclusion holds because $O_p(G)$ is characteristic 
in $G$ for all $G$. Thus $P$ is $\calf$-radical if and only if it is 
$\cale$-radical, and so $\cale^{cr}=\calf^{cr}$, proving (a).

By (c), for each pair $P\le S$ and $Q\in P^\calf$, there are 
$\alpha\in\autf(S)$ and $\varphi\in\Iso_\cale(P,\alpha^{-1}(Q))$, and 
$\alpha^{-1}(Q)\in P^\cale$ is such that $|N_S(\alpha^{-1}(Q))|=|N_S(Q)|$. 
So $P$ is fully normalized in $\calf$ if and only if it is fully 
normalized in $\cale$, and similarly for fully centralized. 
\end{proof}

The following terminology will be useful when working with fusion 
subsystems of index prime to $p$. 

\begin{Defi} \label{d:multiplicative}
Let $\calf$ be a saturated fusion system over a finite $p$-group $S$, and 
let $\calf_0\subseteq\calf$ be a full subcategory with $S\in\Ob(\calf_0)$. 
A map $\theta\:\Mor(\calf_0)\too\Gamma$ to a group $\Gamma$ will be called 
\emph{multiplicative} if it sends composites in $\calf_0$ to products in 
$\Gamma$ and sends inclusions to the identity. It will be called 
\emph{strongly multiplicative} if in addition, $\theta(O^{p'}(\autf(P)))=1$ 
for each $P\in\Ob(\calf_0)$. 
\end{Defi}

Of course, the condition that $\theta$ send composites to products is 
equivalent to saying that it extends to a functor from $\calf_0$ to 
$\calb(\Gamma)$, where $\calb(\Gamma)$ is the category with one object and 
morphism group $\Gamma$.
In these terms, the basic properties of the group $\quotfus\calf$ and of 
saturated fusion subsystems of index prime to $p$ are summarized as 
follows.

\begin{Prop}[{\cite[Theorem I.7.7]{AKO}}] \label{p:thetaF} \fr2
The following hold for each saturated fusion system $\calf$ over a finite 
$p$-group $S$. 
\begin{enuma}

\item There is a unique multiplicative map 
$\theta_\calf\:\Mor(\calf^c)\too\quotfus\calf$ whose restriction to 
$\autf(S)$ is the natural surjection. Also, $\theta_\calf$ is strongly 
multiplicative, and is universal in the following sense: if $\Gamma$ is a 
group and $\theta\:\Mor(\calf^c)\too\Gamma$ is strongly multiplicative, 
then there is a unique homomorphism $f\:\quotfus\calf\too\Gamma$ such that 
$\theta=f\circ\theta_\calf$.

\item There is a bijection 
	\[ \bigl\{ \textup{subgroups of $\quotfus\calf$} \bigr\} 
	\Right5{\cong} 
	\left\{ \parbox{50mm}{\rm saturated fusion subsystems $\cale\le\calf$ 
	of index prime to $p$} \right\} \]
that sends $H\le\quotfus\calf$ to the fusion subsystem 
$\cale_H=\gen{\theta_\calf^{-1}(H)}$. This subsystem has 
the properties that $(\cale_H)^c=\calf^c$ as sets and 
$(\cale_H)^c=\theta_\calf^{-1}(H)$ as categories. In particular, there is a 
unique smallest fusion subsystem $O^{p'}(\calf)=\gen{\theta_\calf^{-1}(1)}$ 
of index prime to $p$, and 
	\[ \Aut_{O^{p'}(\calf)}(S) = \autf^0(S) 
	= \Ker(\theta_\calf|_{\Aut_\calf(S)}). \]

\end{enuma}
\end{Prop}

\begin{proof} Point (a) follows from \cite[Theorem I.7.7(d)]{AKO} 
(where $\quotfus\calf$ is defined to be the target of the universal 
strongly multiplicative map on $\Mor(\calf^c)$). Point (b) is shown in 
\cite[Theorem I.7.7(b,c)]{AKO}. 
\end{proof}

By definition, $\quotfus\calf=1$ if $\autf(S)$ is generated by elements 
$\alpha\in\autf(S)$ such that $\alpha|_P\in O^{p'}(\autf(P))$ for some 
$P\in\calf^c$, and in this case, $O^{p'}(\calf)=\calf$ by Proposition 
\ref{p:thetaF}. In particular, this always applies if $\outf(S)=1$.

There is also a geometric interpretation of universal multiplicative maps. 
For each fusion system $\calf$ over a $p$-group $S$, and each full 
subcategory $\calf_0\subseteq\calf$ containing $S$, let 
$\theta\:\Mor(\calf_0)\too\pi_1(|\calf_0|,*)$ be the map that sends 
$\varphi\in\homf(P,Q)$ to the loop (based at $*=[S]$) formed by the edges 
corresponding to morphisms $\incl_P^S$, $\varphi$, and $\incl_Q^S$. Then 
$\theta$ is the universal multiplicative map on $\Mor(\calf_0)$ (see, e.g., 
\cite[Proposition III.2.8]{AKO}). By \cite[Theorem III.4.19]{AKO}, 
$\theta_\calf$ is universal among all multiplicative maps defined on 
$\Mor(\calf^c)$ (not only those that are strongly multiplicative), and 
hence $\quotfus\calf\cong\pi_1(|\calf^c|)$.

We next recall some standard notation for isomorphisms and automorphisms of 
fusion systems. If $\beta\:S\xto{~\cong~}T$ is an 
isomorphism between finite $p$-groups, and $\calf$ is a fusion system over 
$S$, then we set $\9\beta\calf=\beta\calf\beta^{-1}$: the fusion 
system over $T$ for which $\Hom_{\9\beta\calf}(\beta(P),\beta(Q))$ 
(for $P,Q\le S$) is the set of all composites 
$(\beta|_Q)\circ\varphi\circ(\beta|_P)^{-1}$ for $\varphi\in\homf(P,Q)$. 
When $\cale$ and $\calf$ are fusion systems over finite $p$-groups $T$ 
and $S$, respectively, $\cale\cong\calf$ means that $\cale=\9\beta\calf$ 
for some $\beta\in\Iso(S,T)$. Using this notation, we define, for a fusion 
system $\calf$ over $S$, 
	\[ \Aut(\calf) = \{\beta\in\Aut(S) \,|\, \9\beta\calf=\calf\} 
	\qquad\textup{and}\qquad \Out(\calf)=\Aut(\calf)/\autf(S). \]
This notation is useful when classifying extensions of index 
prime to $p$.

Rather than define normal fusion subsystems here, we note for the purpose 
of the following lemma that if $\cale\le\calf$ has index prime to $p$, 
where both are saturated fusion systems over $S$, then $\cale\nsg\calf$ if 
and only if $\9\alpha\cale=\cale$ for each $\alpha\in\autf(S)$. By 
\cite[Definition I.6.1]{AKO}, there are two conditions we need to check, of 
which the Frattini condition follows from Lemma \ref{l:Ec=Fc}(c), while the 
extension condition holds since $\cale$ and $\calf$ are fusion systems over 
the same $p$-group.

\begin{Lem} \label{F_0<|F_i} \fr2
Let $\calf_0$ be a saturated fusion system over a finite $p$-group $S$. 
Then the map 
	\[ \Psi\: \left\{ \parbox{50mm}{saturated fusion systems $\calf$ 
	over $S$ with $\calf_0\nsg\calf$ normal of index prime to $p$} \right\} 
	\Right5{\cong} 
	\left\{ \parbox{40mm}{subgroups of $\Out(\calf_0)$ of order prime 
	to $p$} \right\} \]
which sends $\calf$ to $\autf(S)/\Aut_{\calf_0}(S)\le\Out(\calf_0)$ is a 
bijection.
\end{Lem}

\begin{proof} If $\calf_1$ and $\calf_2$ are two such extensions of 
$\calf_0$, and $\Aut_{\calf_1}(S)=\Aut_{\calf_2}(S)$, then 
$\calf_1=\calf_2$ since $\calf_i=\gen{\calf_0,\Aut_{\calf_i}(S)}$ (the 
Frattini condition holds for $\calf_0\le\calf_i$ by Lemma 
\ref{l:Ec=Fc}(c)). Thus $\Psi$ is injective. 

If $\pi=\til\pi/\Aut_{\calf_0}(S)\le\Out(\calf_0)$ has order prime to $p$, 
then by \cite[Theorem 
5.7(a)]{BCGLO2}, the fusion system $\calf\defeq\gen{\calf_0,\til\pi}$ is 
saturated and contains $\calf_0$ as a normal subsystem, and 
$\Psi(\calf)=\Gamma/\Aut_{\calf_0}(S)$. So $\Psi$ is onto.
\end{proof}

The injectivity of $\Psi$ in Lemma \ref{F_0<|F_i} implies that if $\calf_1$ 
and $\calf_2$ are two saturated fusion systems over the same $p$-group $S$, 
then $\calf_1=\calf_2$ if $O^{p'}(\calf_1)=O^{p'}(\calf_2)$ and 
$\Aut_{\calf_1}(S)=\Aut_{\calf_2}(S)$. Later, in Lemma \ref{Op'F1=Op'F2}, 
we will replace this by a more general hypothesis.

If $\calf_1$ and $\calf_2$ are fusion systems over finite $p$-groups $S_1$ 
and $S_2$, respectively, then 
	\[ \calf_1\times\calf_2 = \Gen{(\alpha_1,\alpha_2)\in
	\Hom(P_1\times P_2,Q_1\times Q_2) \,\big|\, 
	\alpha_i\in\Hom_{\calf_i}(P_i,Q_i) ~\textup{for $i=1,2$} }. \]
By \cite[Theorem I.6.6]{AKO}, $\calf_1\times\calf_2$ is saturated if 
$\calf_1$ and $\calf_2$ are both saturated.

\begin{Lem} \label{l:F1xF2}
\begin{enuma} 
\item If $G_1$ and $G_2$ are finite groups with $S_i\in\sylp{G_i}$ for 
$i=1,2$, then $\calf_{S_1\times S_2}(G_1\times G_2) = \calf_{S_1}(G_1) 
\times \calf_{S_2}(G_2)$. 

\item For each pair of fusion systems $\calf_1$ and $\calf_2$ over $S_1$ 
and $S_2$, 
	\[ O^{p'}(\calf_1\times\calf_2)=O^{p'}(\calf_1)\times 
	O^{p'}(\calf_2) \quad\textup{and hence}\quad 
	\quotfus{\calf_1\times\calf_2}=\quotfus{\calf_1}\times 
	\quotfus{\calf_2}. \]

\end{enuma}
\end{Lem}

\begin{proof} We leave the proof of (a) as an exercise.

The first equality in (b) is shown in \cite[Proposition 3.4]{AOV1}. In 
particular, 
	\begin{align*} 
	\Aut_{\calf_1\times\calf_2}^0(S_1\times S_2) 
	&= \Aut_{O^{p'}(\calf_1\times\calf_2)}(S_1\times S_2) \\
	&= \Aut_{O^{p'}(\calf_1)}(S_1) \times \Aut_{O^{p'}(\calf_2)}(S_2)
	= \Aut_{\calf_1}^0(S_1)\times\Aut_{\calf_2}^0(S_2) 
	\end{align*}
where the first and third equalities follow from Proposition 
\ref{p:thetaF}(b), and so 
$\quotfus{\calf_1\times\calf_2}=\quotfus{\calf_1}\times 
\quotfus{\calf_2}$.
\end{proof}

The next lemma says that when checking whether a fusion subsystem of 
$\calf$ has index prime to $p$, it suffices to look at $\calf$-centric 
$\calf$-radical subgroups.

\begin{Lem} \label{l:Op'*(F)-cr}
The following hold for each saturated fusion system $\calf$ 
over a finite $p$-group $S$. 
\begin{enuma} 
\item We have $O^{p'}_*(\calf) = \Gen{O^{p'}(\autf(R)) \,\big|\, 
R\in\calf^{cr}}$. 
\item Assume $\calf_0\le\calf$ is a fusion subsystem over $S$ (not 
necessarily saturated) such that 
$\calf\le\gen{\calf_0,\Aut(\calf_0)}$. Then $\calf_0\ge 
O^{p'}_*(\calf)$ and thus has index prime to $p$ in $\calf$. 
\end{enuma}
\end{Lem}

\begin{proof} Assume $\calf_0\le\calf$ is such that 
$\calf\le\gen{\calf_0,\Aut(\calf_0)}$. We claim that for each $P,Q\le 
S$ and each $\varphi\in\homf(P,Q)$, conjugation by $\varphi$ sends 
$\Aut_{\calf_0}(P)$ onto $\Aut_{\calf_0}(\varphi(P))$. This is clear when 
$\varphi\in\Hom_{\calf_0}(P,Q)$, and holds by definition when $\varphi$ is 
the restriction of an element of $\Aut(\calf_0)\le\Aut(S)$. So it holds for 
all $\varphi\in\homf(P,Q)$, since $\varphi$ is a composite of such 
morphisms. 


In particular, $\Aut_{\calf_0}(P)\nsg\autf(P)$. If $P$ is fully normalized in 
$\calf$, then $\Aut_{\calf_0}(P)\ge\Aut_S(P)\in\sylp{\autf(P)}$ since 
$\Aut_{\calf_0}(S)\ge\Inn(S)$, and so 
$\Aut_{\calf_0}(P)$ contains the normal closure $O^{p'}(\autf(P))$ of 
$\Aut_S(P)$ in $\autf(P)$. We just saw that $\Aut_{\calf_0}(Q)$ is isomorphic to 
$\Aut_{\calf_0}(P)$ whenever $Q\in P^\calf$, and hence $\Aut_{\calf_0}(P)\ge 
O^{p'}(\autf(P))$ for all $P\le S$. Thus $\calf_0\ge 
O^{p'}_*(\calf)$, and has index prime to $p$ in $\calf$.

We now apply this to 
$\calf_0\defeq\gen{O^{p'}(\autf(R))\,|\,R\in\calf^{cr}} \le 
O^{p'}_*(\calf)$. Set 
$\cale=\gen{\calf_0,\autf(S)}$: we will show that $\cale=\calf$, and hence 
by the above that $\calf_0\ge O^{p'}_*(\calf)$.  
We first check that $\autf(P)=\Aut_\cale(P)$ for each 
$P\in\calf^{cr}$ fully normalized in $\calf$. Assume otherwise, let $P$ 
a counterexample with $|P|$ maximal, and note that $P<S$. 
By the Frattini argument and since $\Aut_S(P)\in\sylp{\autf(P)}$ by the Sylow 
axiom, we have 
$\autf(P)=O^{p'}(\autf(P))\cdot N_{\autf(P)}(\Aut_S(P))$. Each $\alpha\in 
N_{\autf(P)}(\Aut_S(P))$ extends to some $\4\alpha\in\autf(N_S(P))$ by the 
extension axiom, this is by Theorem \ref{AFT} (Alperin's fusion theorem) 
a composite of restrictions of elements in $\autf(Q)$ for $|Q|\ge|N_S(P)|>|P|$, and 
hence $\alpha\in\Mor(\cale)$ by the maximality assumption. Also, 
$O^{p'}(\autf(P))\subseteq\Mor(\cale)$ by definition of $\cale$, so 
$\autf(P)\subseteq\Mor(\cale)$, contradicting the original assumption. Thus 
	\[ \cale \ge \gen{\autf(P)\,|\,P\in\calf^{cr}~
	\textup{fully normalized in $\calf$}} = \calf, \] 
the last equality by Theorem \ref{AFT} again, and hence 
$\calf=\cale=\gen{\calf_0,\autf(S)}$. So 
$\calf_0\ge O^{p'}_*(\calf)$, and the opposite inclusion is clear.
\end{proof}

Note that by using the stronger form of Alperin's fusion theorem stated in 
\cite[Theorem I.3.5]{AKO}, we could replace $\calf^{cr}$ by the set 
consisting of essential subgroups and $S$.

The following proposition says that restrictions of $\theta_\calf$ to 
certain full subcategories are also universal among \emph{strongly} 
multiplicative maps. 

\begin{Prop} \label{p:pi1(Fcr)}
Let $\calf$ be a saturated fusion system over a finite $p$-group $S$, let 
$\scrp\subseteq\calf^c$ be a family of subgroups closed under 
$\calf$-conjugacy and containing $\calf^{cr}$, and let 
$\calf^\scrp\subseteq\calf^c$ be the full subcategory with object set 
$\scrp$. Then 
	\[ \theta_\calf^{\scrp} = \theta_\calf|_{\Mor(\calf^{\scrp})} 
	\: \Mor(\calf^{\scrp}) \Right5{} \quotfus\calf \]
is universal among all strongly multiplicative maps 
$\Mor(\calf^{\scrp})\too\Gamma$. 
\end{Prop}

\begin{proof} Let $\Gamma$ be the free group with generators 
$[\varphi]$ for all $\varphi\in\Mor(\calf^{\scrp})$, modulo relations 
\begin{itemize} 
\item $[\varphi][\psi]=[\varphi\circ\psi]$ whenever $\varphi\circ\psi$ is 
defined;
\item $[\incl_P^Q]=1$ for all $P\le Q$ in $\scrp$; and 
\item $[\alpha]=1$ for all $\alpha\in O^{p'}(\autf(P))$ with $P\in\scrp$.
\end{itemize}
Let $\theta\:\Mor(\calf^{\scrp})\too\Gamma$ be the natural map sending 
$\varphi$ to $[\varphi]$. Since 
$\theta_\calf^\scrp\:\Mor(\calf^\scrp)\too\quotfus\calf$ is strongly 
multiplicative, there is $f\:\Gamma\too\quotfus\calf$ such that 
$f([\varphi])=\theta_\calf^\scrp(\varphi)$ for each 
$\varphi\in\Mor(\calf^\scrp)$, and hence such that 
$f\circ\theta=\theta_\calf^{\scrp}$. In particular, $f$ is onto since 
$\theta_\calf^\scrp$ is onto.

Set $\cale=\gen{\theta^{-1}(1)}$: the smallest fusion system over $S$ 
containing all morphisms in $\theta^{-1}(1)$. Thus $\cale\ge 
O^{p'}_*(\calf)$ by Lemma \ref{l:Op'*(F)-cr} and since 
$\scrp\supseteq\calf^{cr}$, and so $\cale$ has index prime to $p$ in 
$\calf$. By Lemma \ref{l:Ec=Fc}(b), $\cale$ and $\calf$ have the same fully 
normalized and fully centralized subgroups. By Lemma \ref{l:Ec=Fc}(c), the 
Frattini condition holds: for each $P\le S$ and each 
$\varphi\in\homf(P,S)$, there are $\alpha\in\autf(S)$ and 
$\varphi_0\in\Hom_\cale(P,S)$ such that $\varphi=\alpha\circ\varphi_0$.

By the Frattini condition, each $\varphi\in\Mor(\calf^\scrp)$ is a 
composite of restrictions of morphisms in $\theta^{-1}(1)$ and elements of 
$\autf(S)$. Hence $\theta(\autf(S))=\Gamma$. Since $f$ is onto, to show 
that it is an isomorphism and hence that $\theta_\calf^\scrp$ is universal, 
it remains to show that $\Ker(\theta|_{\autf(S)})\ge\autf^0(S)$. We do 
this by first proving that $\cale$ is saturated.

By construction, $\cale$ is $\scrp$-generated in the sense of 
\cite[Definition 2.1]{BCGLO1}. We claim that $\cale$ is also 
$\scrp$-saturated: that the Sylow and extension axioms hold for members of 
$\scrp$. If $P\in\scrp$ is fully normalized in $\cale$, hence in $\calf$ 
by Lemma \ref{l:Ec=Fc}(b), then $\Aut_S(P)\in\sylp{\Aut_\cale(P)}$ 
since $\Aut_S(P)\in\sylp{\autf(P)}$. Thus the Sylow axiom holds in $\cale$ 
for members of $\scrp$. 

Fix $\varphi\in\Hom_\cale(P,Q)$, where $P,Q\in\scrp$ and $\varphi(P)$ is 
fully centralized in $\cale$ (hence in $\calf$), and let 
$N_\varphi\le N_S(P)$ be as in Definition \ref{d:sfs}(b). (Note that 
$N_\varphi$ is the same, whether we are working in $\cale$ or in $\calf$.) 
By the extension axiom for $\calf$, $\varphi$ extends to some 
$\4\varphi\in\homf(N_\varphi,S)$. By the Frattini condition, 
$\4\varphi=\alpha\circ\varphi_0$, where $\alpha\in\autf(S)$ and 
$\varphi_0\in\Hom_\cale(P,S)$. Thus 
	\[ 1 = \theta(\varphi) = \theta(\alpha\circ\varphi_0|_P) = 
	\theta(\alpha)\cdot\theta(\varphi_0|_P) = \theta(\alpha) \]
where $\theta(\varphi)=\varphi(\varphi_0|_P)=1$ since 
$\varphi$ and $\varphi_0|_P$ are both in $\cale=\gen{\theta^{-1}(1)}$. 
Hence $\alpha\in\Aut_\cale(S)$, and so $\4\varphi\in\Hom_\cale(N_\varphi,S)$. 
The extension axiom for $\cale$ thus holds for members of $\scrp$, and so 
$\cale$ is $\scrp$-saturated. 

If $P\in\calf^c\sminus\scrp$ is fully normalized in $\calf$, then 
	\[ \Aut_S(P)\cap O_p(\Aut_\cale(P)) = \Aut_S(P)\cap O_p(\autf(P)) = 
	O_p(\autf(P)) > \Inn(P) \] 
since $P$ is not $\calf$-radical. Since each $\cale$-conjugacy class in 
$\calf^c\sminus\scrp$ contains such subgroups, $\cale$ satisfies condition 
($*$) in \cite[Theorem 2.2]{BCGLO1}. Since $\cale$ is $\scrp$-generated and 
$\scrp$-saturated, $\cale$ is saturated by that theorem. 
So $\cale\ge O^{p'}(\calf)$ by Proposition \ref{p:thetaF}(b), and hence 
	\beq \Ker(\theta|_{\autf(S)})=\Aut_\cale(S) \ge 
	\Aut_{O^{p'}(\calf)}(S) =  \autf^0(S). \qedhere 
	\eeq
\end{proof}

Proposition \ref{p:pi1(Fcr)} is similar to a theorem of Grodal 
\cite[Theorem 7.5.1]{Grodal}, though that theorem is stated only 
when $\calf$ is the fusion system of a group (and for a slightly different 
definition of radical subgroups). 

The following very elementary family of examples shows that 
$\theta_\calf|_{\calf^{cr}}$ need \emph{not}, in general, be universal 
among all multiplicative maps on $\Mor(\calf^{cr})$: only among those that 
are strongly multiplicative. Recall \cite[Definition I.4.8]{AKO} 
that a saturated fusion system $\calf$ over a finite $p$-group $S$ is 
\emph{constrained} if $O_p(\calf)$ is $\calf$-centric.

\begin{Ex} \label{ex:constr}
Let $\calf$ be a constrained saturated fusion system over $S$, and set 
$Q=O_p(\calf)$. Then $Q\in\calf^{cr}$ and is contained in all other members 
of $\calf^{cr}$. The universal group for multiplicative maps defined on 
$\Mor(\calf^{cr})$ is isomorphic to $\autf(Q)$, while that for strongly 
multiplicative maps is isomorphic to 
$\quotfus\calf\cong\autf(Q)/O^{p'}(\autf(Q))$. 
\end{Ex}

In fact, the universal group for multiplicative maps on $\Mor(\calf^{cr})$ 
need not even be finite.

\begin{Ex} \label{ex:infinite} Set $G=\SL_2(9)$ and $p=2$, and choose 
$S\in\syl2{G}$. Thus $S\cong Q_{16}$ and contains two subgroups $Q_1,Q_2\le 
S$ that are quaternion of order $8$. Set $\calf=\calf_S(G)$; then 
$\calf^{cr}=\{S,Q_1,Q_2\}$, and hence the category $\calf^{cr}$ is the 
union of the two categories $N_\calf(Q_1)^{cr}$ and $N_\calf(Q_2)^{cr}$, 
with intersection $N_\calf(S)^{cr}$. Since these three normalizers are all 
constrained, we see, using Example \ref{ex:constr}, that the universal 
group for multiplicative maps on $\Mor(\calf^{cr})$ is isomorphic to 
	\[ \autf(Q_1)\*_{\autf(S)}\autf(Q_2) \cong 
	\Sigma_4\*_{\!D_8}\Sigma_4. \]
\end{Ex}

We now list some more tools that will be useful when working with fusion 
subsystems of index prime to $p$.

When $\cale$ is a saturated fusion subsystem of $\calf$, $O^{p'}(\cale)$ is 
not, in general, a fusion subsystem of $O^{p'}(\calf)$. The next two lemmas 
describe some cases when this is, in fact, true. 

\begin{Lem} \label{l:Op'F=F:1} \fr2
Let $\calf$ be a saturated fusion system over a finite $p$-group $S$, and 
let $\cale\le\calf$ be a saturated fusion subsystem over $T\le S$. 
\begin{enuma} 

\item If $\cale^{cr}\subseteq\calf^c$ as sets, then $O^{p'}(\cale)\le 
O^{p'}(\calf)$. 

\item If $\cale=C_\calf(U)$ for some abelian subgroup $1\ne U\le S$ fully 
centralized in $\calf$, then $\cale^c\subseteq\calf^c$ as sets, and 
hence $O^{p'}(\cale)\le O^{p'}(\calf)$.

\end{enuma}
\end{Lem}

\begin{proof} \textbf{(a) } Since $\cale^{cr}\subseteq\calf^c$ as 
categories, $\theta_\calf$ restricts to a strongly multiplicative map 
$\Mor(\cale^{cr})\too\quotfus\calf$. By the universal property of 
$\theta_\cale|_{\Mor(\cale^{cr})}$ (Proposition \ref{p:pi1(Fcr)}), there is 
a (unique) homomorphism $\omega\:\quotfus\cale\too\quotfus\calf$ such that 
$\omega\circ\theta_\cale|_{\Mor(\cale^{cr})}
=\theta_\calf|_{\Mor(\cale^{cr})}$. Hence 
	\[ \Mor(O^{p'}(\cale)^{cr}) = (\theta_\cale|_{\Mor(\cale^{cr})})^{-1}(1)
	\le\theta_\calf^{-1}(1) =\Mor(O^{p'}(\calf)^c) \] 
where the first equality holds by Lemma \ref{l:Ec=Fc}(a) 
($O^{p'}(\cale)^{cr}$ and $\cale^{cr}$ have the same objects). 
Since each morphism in $O^{p'}(\cale)$ is a composite of restrictions of 
morphisms in $O^{p'}(\cale)^{cr}$, 
we conclude that $O^{p'}(\cale)\le O^{p'}(\calf)$. 

\smallskip

\noindent\textbf{(b) } Assume $\cale=C_\calf(U)$ and hence $T=C_S(U)$, 
where $U\ne1$ is abelian and fully centralized in $\calf$. Fix 
$P\in\cale^c$; we must show that $P\in\calf^c$. Note that $P\ge C_T(P)\ge 
U$. 

Choose $P^*\in P^\calf$ fully centralized in $\calf$ and 
$\varphi\in\isof(P,P^*)$, and set $U^*=\varphi(U)$. Since $U$ is fully 
centralized in $\calf$, there is $\psi\in\homf(C_S(U^*),T)$ (recall 
$T=C_S(U)$) such that $\psi|_{U^*}=\varphi^{-1}|_{U^*}\in\isof(U^*,U)$. 
Also, $P^*\le C_S(U^*)$ since $P\le C_S(U)$. Set $R=\psi(P^*)$; then 
$\psi\varphi\in\Iso_\cale(P,R)$ since its restriction to $U$ is the 
identity. So $R\in P^\cale\subseteq\cale^c$, hence $C_S(R)\le C_S(U)=T$, 
and thus $C_S(R)=C_T(R)\le R$. But then $C_S(P^*)\le P^*$ since 
$\psi(C_S(P^*))\le C_S(R)$, and $P,P^*\in\calf^c$ since $P^*\in 
P^\calf$ is fully centralized in $\calf$. 
\end{proof}

We will also need to deal with ``subsystems of $p$-power index'' 
in a few cases.

\begin{Defi}[{\cite[Definitions 2.1 \& 3.1]{BCGLO2}}] \label{d:hyp(F)}
Let $\calf$ be a saturated fusion system over a finite $p$-group $S$.
\begin{enuma} 

\item The \emph{hyperfocal subgroup} of $\calf$ is defined as follows:
	\[ \hyp(\calf) = \Gen{ g^{-1}\alpha(g) \,\big|\, g\in P\le S, ~ 
	\alpha\in O^p(\autf(P)) }. \]

\item A fusion subsystem $\cale\le\calf$ over $T\le S$ has \emph{$p$-power 
index} in $\calf$ if $T\ge\hyp(\calf)$, and for all $P\le T$, 
$\Aut_\cale(P)\ge O^p(\autf(P))$. 


\end{enuma}
\end{Defi}

By \cite[Theorem 4.3]{BCGLO2} or \cite[Theorem I.7.4]{AKO}, for each 
saturated fusion system $\calf$ over $S$ and each $T\le S$ containing 
$\hyp(\calf)$, there is a unique saturated fusion subsystem $\calf_T$ over 
$T$ of $p$-power index in $\calf$, and 
$\calf_T=\gen{\Inn(S),O^p(\autf(P))\,|\,P\le T}$ by the second 
reference.

\begin{Lem} \label{l:Op'Op<Op'} \fr2
Let $\cale\le\calf$ be saturated fusion systems over finite $p$-groups 
$T\le S$, where $\cale$ has $p$-power index in $\calf$ and $T$ is 
normal in $S$. Then 
$O^{p'}(\cale)\le O^{p'}(\calf)$, and $\quotfus\calf$ is isomorphic to a 
subquotient of $\quotfus\cale$. 
\end{Lem}

\begin{proof} Set $\calf_0=O^{p'}(\calf)$. Then 
$T\ge\hyp(\calf)\ge\hyp(\calf_0)$, so by \cite[Theorem I.7.4]{AKO}, there 
is a unique saturated fusion subsystem 
	\[ \cale_0 = \Gen{\Inn(T),O^p(\Aut_{\calf_0}(P))\,\big|\, P\le T} 
	\le \calf_0 \] 
over $T$ of $p$-power index in $\calf_0$. Also, $\cale_0\le\cale$ since 
$\cale = \gen{\Inn(T),O^p(\Aut_{\calf}(P))\,|\, P\le T}$. 

We first prove that $O^{p'}(\cale)\le O^{p'}(\calf)$. 
By \cite[Lemma 3.4(b)]{BCGLO2}, 
	\[ \calf = \gen{O^p_*(\calf),\autf(S)} = \gen{O^p_*(\calf),\Inn(S)} 
	\quad\textup{where}\quad O^p_*(\calf) = \gen{O^p(\autf(P))\,|\, 
	P\le S} \]
(as fusion systems over $S$), where the second equality holds since 
$\Inn(S)\in\sylp{\autf(S)}$. Also, 
$T$ is strongly closed in $O^p_*(\calf)$ by definition of $\hyp(\calf)$ and 
since $T\ge\hyp(\calf)$. So $T$ is also strongly closed in 
$\calf=\gen{O^p_*(\calf),\Inn(S)}$ (recall $T\nsg S$).

Let $\calf_0|_T\subseteq\calf_0$ and $\calf|_T\subseteq\calf$ be the full 
subcategories with objects the subgroups of $T$. Also, 
$\calf=\gen{\calf_0,\Aut_\calf(S)}$ by Lemma \ref{l:Ec=Fc}(c) and 
$\calf_0|_T=\gen{\cale_0,\Aut_{\calf_0}(T)}$ by \cite[Lemma 3.4(b)]{BCGLO2} 
again, so 
	\[ \calf|_T = \Gen{\calf_0|_T,\alpha|_T \,\big|\, \alpha\in\autf(S)} 
	= \gen{\cale_0,\Aut_\calf(T)}  \]
since $T$ is strongly closed. Also, $\Aut_\calf(T)$ normalizes $\calf_0|_T$ 
since $\Aut_{\calf_0}(T)$ and $\autf(S)$ both normalize it, and hence 
$\autf(T)$ also normalizes $\cale_0$ by definition of $\cale_0$. Thus 
$\autf(T)\le\Aut(\cale_0)$, so 
$\cale\le\calf|_T\le\gen{\cale_0,\Aut(\cale_0)}$, and $\cale_0$ has index 
prime to $p$ in $\cale$ by Lemma \ref{l:Op'*(F)-cr}. Since $\cale_0$ is 
saturated, $O^{p'}(\cale)\le\cale_0\le \calf_0=O^{p'}(\calf)$.


It remains to show that $\quotfus\calf$ is isomorphic to a 
subquotient of $\quotfus\cale$. Consider the homomorphisms 
	\begin{align*} 
	\quotfus\calf &= \autf(S)/\Aut_{\calf_0}(S) \Right3{f_1} 
	\Aut_\calf(T)/\Aut_{\calf_0}(T) \cong\Out_\calf(T)/\Out_{\calf_0}(T) \\
	\quotfus\cale &= \Out_\cale(T)/\Out_{O^{p'}(\cale)}(T) 
	\Onto3{f_2} \Out_\cale(T)/\Out_{\cale_0}(T) \RIGHT3{f_3}{\cong} 
	\Out_\calf(T)/\Out_{\calf_0}(T) 
	\end{align*} 
where $f_1$ is induced by restriction to $T$ (recall that $T$ is 
strongly closed in $\calf$), $f_2$ is the natural surjection, and $f_3$ 
is induced by the inclusion 
$\Out_\cale(T)\le\outf(T)$. Here, $f_3$ is an isomorphism since 
$\Out_\cale(T)=O^p(\outf(T))$ and $\Out_{\cale_0}(T)=O^p(\Out_{\calf_0}(T))$ 
by the above descriptions of $\cale$ and $\cale_0$, 
both have order prime to $p$, and $\Out_{\calf_0}(T)\ge O^{p'}(\outf(T))$ 
since $\calf_0=O^{p'}(\calf)\ge O^{p'}_*(\calf)$. 

If $\alpha\in\autf(S)$ is such that 
$\alpha|_T\in\Aut_{\calf_0}(T)$, then by the extension axiom (Definition 
\ref{d:sfs}(b)), there is 
$\beta\in\Aut_{\calf_0}(S)$ such that $\beta|_T=\alpha|_T$, so 
$\beta^{-1}\alpha$ is the identity on $T$, induces an automorphism of 
$p$-power order on $S/T$, and hence has $p$-power order and lies in 
$\Inn(S)$. So $f_1$ is injective, and $\quotfus\calf$ is isomorphic to 
a subquotient of $\quotfus\cale$. 
\end{proof}

The next lemma is useful when comparing subsystems of index prime to $p$ in 
a simple group with those in its quasisimple coverings. 

\begin{Lem} \label{mod-Z(F)} \fr2
Let $\calf$ be a saturated fusion system over a finite $p$-group $S$. Then 
for each central subgroup $Z\le Z(\calf)$, we have 
$O^{p'}(\calf/Z)=O^{p'}(\calf)/Z$ and 
$\quotfus{\calf/Z}\cong\quotfus\calf$. Thus if $\calf=\calf_S(G)$ for a 
finite group $G$, then for each $Z\le Z(G)$ (not necessarily a 
$p$-group), $\quotfus{\calf_{SZ/Z}(G/Z)}\cong\quotfus\calf$.
\end{Lem}

\begin{proof} Assume $P\in\calf^{cr}$, and consider the subgroup 
$\Delta=\{\alpha\in\autf(P)\,|\,[\alpha,P]\le Z\}$. All elements of 
$\autf(P)$ restrict to the identity on $Z$, so $\Delta\nsg\autf(P)$, and 
each $\alpha\in\Delta$ induces the identity on $Z$ and on $S/Z$ and hence 
has $p$-power order \cite[Corollary 5.3.3]{Gorenstein}. Thus $\Delta\le 
O_p(\autf(P))=\Inn(P)$ since $P$ is $\calf$-radical. So if $gZ\in 
C_{S/Z}(P/Z)$, then $c_g\in\Inn(P)$, and $g\in C_S(P)P=P$. Since this 
applies to all members of $P^\calf$, we see that $P/Z\in(\calf/Z)^c$. 

Consider the following diagram: 
	\[ \vcenter{\xymatrix@C=35pt{ 
	\Mor(\calf^{cr}) \ar@{->>}[r]^-{\theta_\calf^{cr}} \ar[d]_{P\mapsto P/Z} & 
	\quotfus\calf \ar@{-->>}[d]^{f} \\ 
	\Mor((\calf/Z)^c) \ar@{->>}[r]^-{\theta_{\calf/Z}} & \quotfus{\calf/Z} 
	}} \] 
where $\theta_\calf^{cr}$ is the restriction of $\theta_\calf$, and 
where $f$ exists (making the square commute) by the universality of 
$\theta_\calf^{cr}$ (Proposition \ref{p:pi1(Fcr)}). This commutativity 
implies that $(O^{p'}(\calf))^{cr}/Z\le O^{p'}(\calf/Z)$, and hence that 
$O^{p'}(\calf)/Z\le O^{p'}(\calf/Z)$ by Alperin's fusion theorem (Theorem 
\ref{AFT}). 

For each $P\le S$ such that $P\ge Z$ and $P/Z$ is fully normalized in 
$\calf/Z$, $\Aut_{O^{p'}(\calf)/Z}(P/Z)$ is normal in $\Aut_{\calf/Z}(P/Z)$ 
and contains $\Aut_{S/Z}(P/Z)\in\sylp{\Aut_{\calf/Z}(P/Z)}$, and thus 
contains $O^{p'}(\Aut_{\calf/Z}(P/Z))$. So $\Aut_{O^{p'}(\calf)/Z}(P/Z)\ge 
O^{p'}(\Aut_{\calf/Z}(P/Z))$ for all $P/Z\le S/Z$, and $O^{p'}(\calf)/Z$ 
has index prime to $p$ in $\calf/Z$. By Proposition \ref{p:thetaF}(b), 
$O^{p'}(\calf)/Z \geq O^{p'}(\calf/Z)$.

We next show that $f$ is an isomorphism. Let $\alpha\in\autf(S)$ be such 
that $\theta_\calf(\alpha)\in\Ker(f)$, and let 
$\alpha'\in\Aut_{\calf/Z}(S/Z)$ be the induced automorphism of 
$S/Z$. Thus 
	\[ \alpha' \in \theta_{\calf/Z}^{-1}(1) =
	\Aut_{O^{p'}(\calf/Z)}(S/Z) = \Aut_{O^{p'}(\calf)/Z}(S/Z), \]
so there is $\beta\in\Aut_{O^{p'}(\calf)}(S)$ such that $\beta^{-1}\alpha$ 
induces the identity on $S/Z$. But $\beta^{-1}\alpha$ also induces the 
identity on $Z\le Z(\calf)$, so it has $p$-power order, and 
$\theta_\calf(\alpha)=\theta_\calf(\beta)=1$ since $\quotfus\calf$ has 
order prime to $p$. This proves that $\Ker(f)=1$, and hence that $f$ 
is an isomorphism.

The last statement is now immediate.
\end{proof}

The next lemma deals with normal subgroups of subsystems of index 
prime to $p$.

\begin{Lem} \label{l:QnsgEnsgF} \fr2
Assume $\cale\le\calf$ are saturated fusion systems over $S$, where $\cale$ 
has index prime to $p$ in $\calf$. Then for $Q\nsg S$, $Q\nsg\calf$ if and 
only if $Q\nsg\cale$ and $Q$ is weakly closed in $\calf$. 
\end{Lem}

\begin{proof} Assume $Q\nsg\cale$ and $Q$ is weakly closed in $\calf$. By 
the Frattini condition (Lemma \ref{l:Ec=Fc}(c)), for each 
$\varphi\in\Hom_\calf(P,R)$, there are $\alpha\in\Aut_{\calf}(S)$ and 
$\varphi_0\in\Hom_{\cale}(P,\alpha^{-1}(R))$ such that 
$\varphi=(\alpha|_{\alpha^{-1}(R)})\circ\varphi_0$. Then $\alpha(Q)=Q$ 
since $Q$ is weakly closed, and $\varphi_0$ extends to a morphism 
$\4\varphi_0\in\Hom_{\cale}(PQ,\alpha^{-1}(RQ))$ such that 
$\4\varphi_0(Q)=Q$ since $Q\nsg\cale$. So 
$(\alpha|_{\alpha^{-1}(RQ)})\circ\4\varphi_0\in\homf(PQ,RQ)$ extends 
$\varphi$. This proves that $Q\nsg\calf$, and the converse is clear. 
\end{proof}

The next lemma, which gives one criterion for proving that $O^{p'}(\calf)$ 
is simple, will be applied in Section \ref{s:simplicity}. 

\begin{Lem} \label{l:no.str.cl.}
Fix a prime $p$ and a finite $p$-group $S$, let $\calf$ be a saturated 
fusion system over $S$, and assume that no nontrivial proper subgroup of 
$S$ is strongly closed in $\calf$. Then either $O^{p'}(\calf)$ is simple; 
or there are subgroups $T_1,\dots,T_k<S$ (for $2\le 
k\le|\quotfus\calf|$) that are strongly closed in $O^{p'}(\calf)$ and 
$\autf(S)$-conjugate to each other, and such that $S=T_1\times\dots\times 
T_k$. In particular, $\calf$ is simple if $\calf=O^{p'}(\calf)$.
\end{Lem}

\begin{proof} Assume $O^{p'}(\calf)$ is not simple, and let $1\ne\cale\nsg 
O^{p'}(\calf)$ be a nontrivial proper normal subsystem over $1\ne T\nsg S$. 
In particular, $T$ is strongly closed in $O^{p'}(\calf)$ (one of the 
conditions for $\cale$ to be normal). If $T=S$, then $\cale$ has index 
prime to $p$ in $O^{p'}(\calf)$ by \cite[Lemma 1.26]{AOV1}, which is 
impossible since $\cale<O^{p'}(\calf)$ and 
$O^{p'}(O^{p'}(\calf))=O^{p'}(\calf)$. So $T<S$. 

Thus there are proper nontrivial subgroups of $S$ strongly closed in 
$O^{p'}(\calf)$, and so $O^{p'}(\calf)<\calf$ by assumption. Let 
$1\ne U<S$ be minimal among all such subgroups. Set 
$m=|\quotfus\calf|=|\autf(S)/\Aut_{O^{p'}(\calf)}(S)|$ (Proposition 
\ref{p:thetaF}), choose coset representatives 
$\alpha_1,\dots,\alpha_m\in\autf(S)$ for $\Aut_{O^{p'}(\calf)}(S)$, and set 
$U_i=\alpha_i(U)$. Then each $U_i$ is a minimal strongly closed subgroup in 
$O^{p'}(\calf)$, and $\5U=U_1\cdots U_m$ is strongly closed in 
$O^{p'}(\calf)$ by \cite[Theorem 2]{A-gfit} or \cite[Theorem 5.22]{Craven}. 
Also, $\5U$ is normalized by $\autf(S)$ by construction. Since each 
morphism in $\calf$ is the composite of a morphism in $O^{p'}(\calf)$ 
followed by the restriction of some element of $\autf(S)$, $\5U$ is 
strongly closed in $\calf$. So $\5U=S$ by assumption. 

For each $1\le i\le m-1$, $U_1\cdots U_i$ is strongly closed in 
$O^{p'}(\calf)$ by \cite[Theorem 2]{A-gfit} again, so $(U_1\cdots U_i)\cap 
U_{i+1}$ is also strongly closed. So by the minimality of $U$ among 
strongly closed subgroups, either $(U_1\cdots U_i)\cap U_{i+1}=1$ or 
$U_{i+1}\le U_1\cdots U_i$. Hence there is a subset 
$I\subseteq\{1,\dots,m\}$ such that $S$ is the direct product of the $U_i$ 
for $i\in I$, and $|I|\ge2$ since $U<S$. 
\end{proof}


\section{Fusion systems with centric and weakly closed subgroups}
\label{s:w.cl.}

We saw in the last section that when $\calf$ is a saturated fusion system 
over a finite $p$-group $S$, its subsystems and extensions of index prime 
to $p$ are determined by their automizers on $S$. In this section, we show 
that in fact, it suffices to consider the automizers of some subgroup 
$A\nsg S$ that is $\calf$-centric and weakly closed in $\calf$. In 
particular, this works in many cases when $A$ is the $p$-power torsion in 
the maximal torus of a finite group of Lie type in defining characteristic 
different from $p$, and allows us to work with $\autf(A)$ rather than 
$\autf(S)$.

We begin with three lemmas that describe the relationship between 
$\autf(A)$ and $\quotfus\calf$ for such a subgroup $A\nsg S$, and show how 
this can be used to get upper and lower bounds for $|\quotfus\calf|$.

\begin{Lem} \label{theta-F-A} \fr3
Let $\calf$ be a saturated fusion system over a finite $p$-group $S$, and 
assume $A\nsg S$ is a subgroup that is $\calf$-centric and weakly closed in 
$\calf$. Let 
	\[ \thet\calf{A} \: \autf(A) \Right5{} \quotfus\calf \]
be the restriction of $\theta_\calf\:\Mor(\calf^c)\too\quotfus\calf$ to 
$\autf(A)$. Then 
\begin{enuma} 
\item $\thet\calf{A}$ is surjective and
$\Ker(\thet\calf{A})=\Aut_{O^{p'}(\calf)}(A)$;
\item $\quotfus\calf=1$ if $O^{p'}(\autf(A))=\autf(A)$; and 
\item for $\cale\le\calf$ of index prime to $p$, $\cale=\calf$ if and only 
if $\Aut_\cale(A)=\autf(A)$.
\end{enuma}
\end{Lem}

\begin{proof} Since $A$ is weakly closed, each $\alpha\in\autf(S)$ 
restricts to an element of $\autf(A)$, and so the surjectivity of 
$\thet\calf{A}$ follows from that of $\theta_\calf|_{\autf(S)}$ 
(Proposition \ref{p:thetaF}(a)). Point (b) is then immediate. 

Set $\cale_H=\gen{\theta_\calf^{-1}(H)}\le\calf$ for each subgroup 
$H\le\quotfus\calf$. Then $\Aut_{\cale_H}(A)=(\thet\calf{A})^{-1}(H)$, 
since each $\alpha\in\Aut_{\cale_H}(A)$ is a composite of restrictions of 
morphisms in $\theta_\calf^{-1}(H)$ (and inclusions are sent to the 
identity). In particular, $\Aut_{O^{p'}(\calf)}(A)=\Ker(\thet\calf{A})$ 
(the case $H=1$), finishing the proof of (a). For each $\cale\le\calf$ of 
index prime to $p$, $\cale=\cale_H$ for some $H\le\quotfus\calf$ by 
Proposition \ref{p:thetaF}(b), and from the surjectivity of 
$\thet\calf{A}$, we see that $\Aut_\cale(A)=\autf(A)$ implies 
$H=\quotfus\calf$, and hence $\cale=\calf$. 
\end{proof}

One immediate consequence of Lemma \ref{theta-F-A} is the following 
upper bound for $|\quotfus\calf|$.

\begin{Lem} \label{l:Op'F=F:2} \fr3
Let $\calf$ be a saturated fusion system over a finite $p$-group $S$, and 
let $\cale\le\calf$ be a saturated fusion subsystem over $T\le S$. 
Assume that $O^{p'}(\cale)\le O^{p'}(\calf)$, that $A\nsg S$ is 
$\calf$-centric and weakly closed in $\calf$, and also that $A\le T$. Then 
$\Aut_{O^{p'}(\cale)}(A)$ is contained in the kernel of the homomorphism 
$\thet\calf{A}$ from $\autf(A)$ onto $\quotfus\calf$, and hence 
	\[ |\quotfus\calf| \le \bigl| \autf(A) : 
	O^{p'}(\autf(A))\cdot\Aut_{O^{p'}(\cale)}(A) \bigr|. \]
In particular, $O^{p'}(\calf)=\calf$ if 
$\autf(A)=O^{p'}(\autf(A))\cdot\Aut_{O^{p'}(\cale)}(A)$. 
\end{Lem}

\begin{proof} Since $O^{p'}(\cale)\le O^{p'}(\calf)$, we have 
$\Aut_{O^{p'}(\cale)}(A) \le \Aut_{O^{p'}(\calf)}(A) = 
\Ker(\thet\calf{A})$ by Lemma 
\ref{theta-F-A}(a). The other claims now follow from the surjectivity of 
$\thet\calf{A}$ and since $p\nmid|\quotfus\calf|$.
\end{proof}

If $A$ is an \emph{abelian} subgroup that is $\calf$-centric and 
weakly closed, then we can say more. The next lemma will be our main tool 
for getting lower bounds on $|\quotfus\calf|$ when it is nontrivial.

\begin{Lem} \label{p:p'-index} \fr4
Let $\calf$ be a saturated fusion system over the $p$-group $S$, and let 
$A\nsg S$ be an abelian subgroup which is $\calf$-centric and weakly closed 
in $\calf$. Set
	\[ X = \{ t\in A \,|\, t^\calf\subseteq A \}, \]
and assume $X\cap Z(S)\ne\{1\}$. Choose $1\ne Z\le\gen{X\cap Z(S)}$, set 
$\cale=C_\calf(Z)$ and $\cale_0=O^{p'}(\cale)$, and let $K_0\le 
K\nsg\autf(A)$ be the normal closures of $\Aut_{\cale_0}(A)$ and 
$\Aut_\cale(A)$ in $\autf(A)$. Then the following hold.
\begin{enuma} 

\item There is a surjective homomorphism $f\:\quotfus\calf\too\autf(A)/K$ 
with the following property: for each morphism $\varphi\in\homf(P,Q)$ in 
$\calf^c$, there is $\alpha\in\autf(A)$ such that $\alpha|_Z=\varphi|_Z$, 
and for each such $\alpha$, $f(\theta_\calf(\varphi))=\alpha K$. 

\item We have $O^{p'}(\autf(A))K_0 \le \Ker(\thet\calf{A}) \le K$.

\end{enuma}
\end{Lem}

\begin{proof} Since $\cale_0\le O^{p'}(\calf)$ by Lemma 
\ref{l:Op'F=F:1}(b), 
$\Aut_{\cale_0}(A)\le\Aut_{O^{p'}(\calf)}(A)=\Ker(\thet\calf{A})$, and 
hence $K_0\le\Ker(\thet\calf{A})$. Since $\Im(\thet\calf{A})$ has order 
prime to $p$, this proves the first inclusion in (b).


Before proving the other claims, we check that 
	\beqq \textup{for each $P\le A$ and $\varphi\in\homf(P,A)$, 
	$\varphi$ extends to some $\4\varphi\in\autf(A)$.} \label{e:Aut(A)} 
	\eeqq
To see this, set $Q=\varphi(P)\le A$, choose $R\in P^\calf=Q^\calf$ that is 
fully centralized in $\calf$, and fix $\chi\in\isof(Q,R)$. By the extension 
axiom, $\chi$ extends to $\4\chi\in\homf(C_S(Q),S)$ and $\chi\varphi$ 
extends to $\4\varphi\in\homf(C_S(P),S)$, and $\4\chi(A)=A=\4\varphi(A)$ 
since $A$ is weakly closed. Then $\4\chi^{-1}\4\varphi|_A\in\autf(A)$, and 
$(\4\chi^{-1}\4\varphi)|_P=\chi^{-1}(\chi\varphi)=\varphi$. 

Let $\theta_0:\Aut_{\calf}(A)\too\autf(A)/K$ be the natural map. We 
first extend $\theta_0$ to a strongly multiplicative map defined on 
$\Mor(\calf^c)$ (see Definition \ref{d:multiplicative}). 

Fix $\varphi\in\Mor_{\calf}(P,Q)$, where $P,Q\in\calf^c$. In particular, 
$P,Q\ge Z(S)\ge Z$, and $\varphi(Z)\le\varphi(\gen{X\cap Z(S)}) 
\le A$ by definition of $X$. By \eqref{e:Aut(A)}, $\varphi|_Z$ is 
the restriction of some automorphism $\alpha\in\autf(A)$. If 
$\alpha'\in\autf(A)$ is another automorphism for which 
$\alpha'|_Z=\varphi|_Z$, then $\alpha^{-1}\alpha'\in C_{\autf(A)}(Z)\le K$, 
and so we can define $\theta(\varphi)=\theta_0(\alpha)=\theta_0(\alpha')$. 
We thus have a well defined map of sets 
$\theta\:\Mor(\calf^c)\too\autf(A)/K$ with $\theta|_{\autf(A)}=\theta_0$, 
and $\theta$ is surjective since $\theta_0$ is surjective. If 
$\varphi$ is an inclusion, then we can choose $\alpha=\Id_A$, and hence 
$\theta(\varphi)=1$. Also, $\autf(A)/K$ has order prime to $p$ since 
$K\ge\Aut_S(A)\in\sylp{\autf(A)}$, and hence $\theta(O^{p'}(\autf(P)))=1$ 
for each $P\in\calf^c$. 

It remains to check that $\theta$ sends composites in $\Mor(\calf^c)$ to 
products in $\autf(A)/K$.  Assume $\varphi\in\Hom_\calf(P,Q)$ and 
$\psi\in\Hom_\calf(Q,R)$, where $P,Q,R\in\calf^c$, and set $P^*=\gen{P\cap 
X}$ and $Q^*=\gen{Q\cap X}$.  Then $\varphi(P\cap X)\subseteq Q\cap X$, so 
$\varphi(P^*)\le Q^*$. By \eqref{e:Aut(A)}, there are  
$\alpha,\beta\in\Aut_\calf(A)$ 
which are extensions of $\varphi|_{P^*}$ and $\psi|_{Q^*}$, respectively. 
Then $\beta\circ\alpha$ is an extension 
of $(\psi\circ\varphi)|_{P^*}$, hence of $(\psi\circ\varphi)|_Z$, and thus 
$\theta(\psi\circ\varphi) =\beta\circ\alpha=\theta(\psi)\circ\theta(\varphi)$.  

Since $\theta_\calf$ is universal among strongly multiplicative maps 
(Proposition \ref{p:thetaF}(a)), there is a unique homomorphism 
$f\:\quotfus\calf\too\autf(A)/K$ such that $\theta=f\circ\theta_\calf$, 
and $f$ is surjective since $\theta$ is surjective. 
This proves (a). Also, $K=\Ker(\theta_0)\ge\Ker(\thet\calf{A})$, which 
proves the second inclusion in (b). 
\end{proof}

We next note that two extensions of index prime to $p$ of the same fusion 
system over $S$ are equal, or at least isomorphic, if they have the same 
automizer on some centric and weakly closed subgroup $A\nsg S$. 

\begin{Lem} \label{Op'F1=Op'F2} \fr3
Let $\calf_1$ and $\calf_2$ be saturated fusion systems over the same 
finite $p$-group $S$. Assume that $A\nsg S$ is $\calf_i$-centric and weakly 
closed in $\calf_i$ for $i=1,2$, and also that 
	\[ O^{p'}(\calf_1)=O^{p'}(\calf_2) \qquad\textup{and}\qquad
	\Aut_{\calf_1}(A)=\Aut_{\calf_2}(A). \]
Then there is $\beta\in\Aut(O^{p'}(\calf_2))$ such that $\beta|_A=\Id_A$ 
and $\calf_1=\9\beta\calf_2$. If in addition, 
$H^1(\Out_{O^{p'}(\calf_1)}(A);Z(A))=0$, then $\calf_1=\calf_2$.
\end{Lem}

\begin{proof} Set $\calf_0=O^{p'}(\calf_1)=O^{p'}(\calf_2)$, 
and consider the subgroup 
	\[ H = \bigl\{\beta\in\Aut(S) \,\big|\,\beta|_A=\Id\bigr\} . 
	\]
Each $\beta\in H$ induces the identity on $S/A$ since $C_S(A)\le A$, 
so $H$ is a $p$-group by, e.g., \cite[Corollary 
5.3.3]{Gorenstein}. Each $\alpha\in\Aut_{\calf_i}(S)$ (for $i=1,2$) 
normalizes $A$ since $A$ is weakly closed, and hence also normalizes 
$H$. For each $\alpha\in\Aut_{\calf_1}(S)$, 
$\alpha|_A\in\Aut_{\calf_1}(A)=\Aut_{\calf_2}(A)$ and normalizes 
$\Aut_S(A)$, and hence extends to some $\alpha'\in\Aut_{\calf_2}(S)$ by the 
extension axiom (Definition \ref{d:sfs}(b)). Then $\alpha'\in\alpha H$, 
and this together with a similar argument in the opposite direction shows 
that $\Aut_{\calf_1}(S)H=\Aut_{\calf_2}(S)H$. Set 
$G=\Aut_{\calf_1}(S)H=\Aut_{\calf_2}(S)H$.

For $i=1,2$, $\Inn(S)H\nsg\Aut_{\calf_i}(S)H$ since 
$\Inn(S)\nsg\Aut_{\calf_i}(S)$ and both normalize $H$. Also, 
$\Aut_{\calf_i}(S)H/\Inn(S)H\cong\Aut_{\calf_i}(S)/\Inn(S)$ has 
order prime to $p$ by the Sylow axiom (Definition \ref{d:sfs}(b) again). So 
by the Schur-Zassenhaus theorem (see \cite[Theorem 6.2.1(i)]{Gorenstein}), 
all splittings of the extension
	\[ 1 \too \Inn(S)H/\Inn(S) \Right4{} G/\Inn(S) \Right4{} G/\Inn(S)H 
	\too 1 \]
are conjugate. Since $\Aut_{\calf_1}(S)/\Inn(S)$ and 
$\Aut_{\calf_2}(S)/\Inn(S)$ are two such splittings (since 
$\Aut_{\calf_i}(S)\cap H\le\Inn(S)$ by the Sylow axiom), there is 
$\beta\in H$ such that $\9\beta(\Aut_{\calf_2}(S))=\Aut_{\calf_1}(S)$. 
Thus $\beta|_A=\Id_A$ and 
	\[ \Aut_{\calf_1}(S) = \9\beta(\Aut_{\calf_2}(S)) = 
	\Aut_{\9\beta\calf_2}(S), \]
and so $\calf_1=\9\beta\calf_2$ by Lemma \ref{F_0<|F_i} (and since 
$O^{p'}(\calf_1)=\calf_0=O^{p'}(\calf_2)$).

Now assume that $H^1(\Out_{\calf_0}(A);Z(A))=0$, and let $G$ be a model for 
$N_{\calf_0}(A)$ (see \cite[Theorem I.4.9]{AKO}). Thus $A\nsg G$, 
$S\in\sylp{G}$, $O_{p'}(G)=1$, and $\calf_S(G)=N_{\calf_0}(A)$. Then 
$C_G(A)\le A$, so $G/A\cong\Out_{\calf_0}(A)$. Since 
$\beta\in\Aut(\calf_0)\le\Aut(S)$ and $\beta|_A=\Id_A$, $\beta$ is also an 
automorphism of $N_{\calf_0}(A)$ and extends to an automorphism $\5\beta\in 
\Aut(G)$ by the uniqueness condition in the model theorem \cite[Lemma 
II.4.3(a)]{AKO}. Also, $\5\beta$ induces the identity on $G/A$ since 
$\5\beta|_A=\Id$ and $C_G(A)\le A$, and $\5\beta$ is conjugation by some 
$a\in Z(A)$ since $H^1(\Out_{\calf_0}(A);Z(A))=0$ (see, e.g., \cite[Lemma 
1.2]{OV2}). But then $\beta\in\Inn(S)\le\Aut(\calf_2)$, and so 
$\calf_1=\9\beta\calf_2=\calf_2$. 
\end{proof}

The last lemma implies as a special case that when $G$ is a finite 
group of Lie type in characteristic different from $p$, adding diagonal 
automorphisms of order prime to $p$ doesn't change its $p$-fusion system.

\begin{Lem} \label{diag-aut} \fr3
Let $p$ be a prime, and let $G\nsg\5G$ be finite groups such that 
$p\nmid|O^p(\5G/G)|$. Choose $\5S\in\sylp{\5G}$, set $S=\5S\cap 
G\in\sylp{G}$, and set $\calf=\calf_S(G)$ and $\5\calf=\calf_{\5S}(\5G)$. 
Assume $A\nsg S$ is $\calf$-centric and weakly closed in $\calf$, and is 
such that $\5G=GC_{\5G}(A)$. Then $\calf$ is normal of $p$-power index in 
$\5\calf$, and $\calf=\5\calf$ if $p\nmid|\5G/G|$. 
\end{Lem}

\begin{proof} Set $\4G=G\cdot O^p(\5G)$. Then $S\in\sylp{\4G}$ since 
$p\nmid|\4G/G|$, and $\calf_S(\4G)$ is normal of $p$-power index in 
$\5\calf$. It remains to show that $\calf=\calf_S(\4G)$. Since $\calf$ is 
normal of index prime to $p$ in $\calf_S(\4G)$, it suffices by Lemma 
\ref{theta-F-A}(c) to show that $\autf(A)=\Aut_{\4G}(A)$, and this holds 
since our original assumption $\5G=GC_{\5G}(A)$ implies that 
$N_{\4G}(A)=N_G(A)C_{\4G}(A)$. 
\end{proof}


\section{Simplicity of fusion systems of known simple groups}
\label{s:simplicity}

We are now ready to explicitly describe $O^{p'}(\calf)$ and 
$\quotfus\calf$, when $\calf$ is the fusion system of a known simple 
group and $S\nnsg\calf$. 

When $r$ is a prime, let $\Lie(r)$ denote the class of finite groups 
of Lie type in defining characteristic $r$. Since this will be used only 
when restricting attention to simple groups, we needn't go into details 
as to exactly which groups are included in this class.

In most cases, the simplicity of $\calf$ was determined by Aschbacher in 
\cite[Chapter 16]{A-gfit}. His conclusions are summarized in the following 
proposition.

\begin{Prop} \label{t:tame.simple}
Fix a prime $p$, a known simple group $G$, and $S\in\sylp{G}$. Set 
$\calf=\calf_S(G)$, and assume that $S\nnsg\calf$. 
\begin{enuma} 

\item If $G\cong A_n$, then $O^{p'}(\calf)=\calf$ if $n\equiv0,1$ (mod 
$p$), and $|\quotfus\calf|=2$ otherwise. In all cases, $O^{p'}(\calf)$ is 
simple, and is realized by $A_{n'}$, where $n'=p\cdot[\frac{n}p]$. 

\item If $G\in\Lie(p)$, or if $p=2$ and $G\cong\lie2F4(2)'$, then $G$ has Lie 
rank at least two, $O^{p'}(\calf)=\calf$, and $\calf$ is simple.

\item If $G$ is a sporadic simple group, then either 
$\calf$ is simple, or $O^{p'}(\calf)$ has index $2$ in $\calf$ and is 
simple, and is realized by a known simple group $H$. See Table 
\ref{tbl:sporlist} for details in the individual cases.

\end{enuma}
\end{Prop}

\renewcommand{\ab}{\textup{ab}}
\newcommand{\es}[1]{\fbox{$#1$}}
\newcommand{\simp}{\textup{simple}}
\newcommand{\asimp}{}
\newcommand{\Gs}{S\nsg\calf}

\begin{table}[ht] 
\renewcommand{\arraystretch}{1.2} \setlength{\arraycolsep}{3mm}
\[ \begin{array}{c|cccccc}
~G~ & ~p=2~ & ~p=3~ & ~p=5~ & ~p=7~ & ~p=11~ & ~p=13~ \\\hline
M_{11} & \simp & \ab & \ab &  & \ab &  \\ 
M_{12} & \simp & \es{\simp} & \ab &  & \ab &  \\ 
M_{22} & \simp & \ab & \ab & \ab & \ab &  \\ 
M_{23} & \simp & \ab & \ab & \ab & \ab &  \\ 
M_{24} & \simp & \es{\asimp M_{12}:2} & \ab & \ab & \ab &  \\ 
J_1 & \ab & \ab & \ab & \ab & \ab &  \\ 
J_2 & \simp & \es{\Gs} & \ab & \ab &  &  \\ 
J_3 & \simp & \Gs & \ab &  &  &  \\ 
J_4 & \simp & \es{\asimp \lie2F4(2)':2} & \ab & \ab & \es{\Gs} &  \\ 
\Co_3 & \simp & \simp & \es{\Gs} & \ab & \ab &  \\ 
\Co_2 & \simp & \simp & \es{\Gs} & \ab & \ab &  \\ 
\Co_1 & \simp & \simp & \simp & \ab & \ab &  \\ 
\HS & \simp & \ab & \es{\Gs} & \ab & \ab &  \\ 
\McL & \simp & \simp & \es{\Gs} & \ab & \ab &  \\ 
\Suz & \simp & \simp & \ab & \ab & \ab &  \\ 
\He & \simp & \es{\asimp M_{12}:2} & \ab & \es{\simp} &  &  \\ 
\Ly & \simp & \simp & \simp & \ab & \ab &  \\ 
\Ru & \simp & \es{\asimp \lie2F4(2)':2} & \es{\asimp \SL_3(5):2} & \ab &  & \ab \\ 
\ON & \simp & \ab & \ab & \es{\simp} & \ab &  \\ 
\Fi_{22} & \simp & \simp & \ab & \ab & \ab & \ab \\ 
\Fi_{23} & \simp & \simp & \ab & \ab & \ab & \ab \\ 
\Fi_{24}' & \simp & \simp & \ab & \es{\simp} & \ab & \ab \\ 
F_5 & \simp & \simp & \simp & \ab & \ab &  \\ 
F_3 & \simp & \simp & \es{\simp} & \ab &  & \ab \\ 
F_2 & \simp & \simp & \simp & \ab & \ab & \ab \\ 
F_1 & \simp & \simp & \simp & \simp & \ab & \es{\simp} 
\end{array} \]
\caption{\small{In all cases, $G$ is a sporadic simple group, 
$S\in\sylp{G}$, and $\calf=\calf_S(G)$. An entry ``$H:2$'' means that 
$O^{p'}(\calf)$ is simple, has index $2$ in $\calf$, and is realized by 
the simple group $H$. A blank entry means that $p\nmid|G|$, and ``\ab'' 
means that $S$ is abelian. A box indicates that $S$ is extraspecial of 
order $p^3$ and exponent $p$.}}
\label{tbl:sporlist}
\end{table}

\begin{proof} \noindent\textbf{(a) } See \cite[16.5]{A-gfit} or 
\cite[Proposition 4.9]{AOV1}. 

\smallskip

\noindent\textbf{(b) } Assume $G\in\Lie(p)$, or $p=2$ and 
$G\cong\lie2F4(2)'$. Then $G$ has Lie rank at least 2 by \cite[Theorem 
15.6.b]{A-gfit}, and so $\calf$ is simple by \cite[16.3]{A-gfit}. 

\smallskip

\noindent\textbf{(c) } Assume $G$ is sporadic. If $p=2$, then 
$G\not\cong J_1$ since $S\nnsg\calf$, and hence $\calf$ is simple in all 
cases by \cite[16.8]{A-gfit}. 
If $p$ is odd, then since $S\nnsg\calf$, $G$ is not 
``$p$-Goldschmidt'' in the terminology of \cite{A-gfit}, and hence is not 
one of the groups listed in \cite[Theorem 15.6]{A-gfit}. By 
\cite[16.10]{A-gfit}, either $\calf$ is simple, or $(G,p)$ is one of the 
pairs listed there (and marked as such in Table \ref{tbl:sporlist}), and 
$\calf$ is realized by the almost simple group given in the same reference. 
In particular, when $\calf$ is not simple, $O^{p'}(\calf)$ is realized by 
one of the simple groups $M_{12}$ or $\lie2F4(2)'$ (when $p=3$) or 
$\SL_3(5)$ (when $p=5$), and hence is simple by 
\cite[16.10]{A-gfit} again or by (b), respectively. 

Note, however, the following two errors in the statement of \cite[Theorem 
15.6]{A-gfit}. When $p=3$ and $G\cong\He$, $\calf$ is isomorphic to the 
fusion system of $M_{24}$ and hence that of $\Aut(M_{12})$: this follows 
from the tables on pp. 47--48 of \cite{RV}. The error in the proof in 
\cite{A-gfit} came from mis-identifying $D=N_G(S)$ (see line 3 on p. 102 in 
\cite{A-gfit}). 

When $p=5$ and $G\cong\Co_1$, $\calf$ is simple by Theorem 2.8 and Table 
2.2 in \cite{indp3}. In this case, the error in the proof of 
\cite[16.9.5]{A-gfit} occurred because the $\calf$-radical subgroups of 
order $5^2$ were overlooked. 

In addition, when $p=3$ and $G\cong\Fi_{24}'$, argument (ii) on p. 101 in 
\cite{A-gfit} cannot be used to prove that $O^{3'}(\calf)=\calf$, since 
there is no maximal subgroup $H\le G$ such that $O_3(H)=F^*(H)$, 
$H=O^{3'}(H)$, and $H\ge N_G(S)$. Instead, argument (iii) (on the same 
page) can be applied. 
\end{proof}

It remains to consider the cases where $G\in\Lie(q)$ for $q\ne p$, using 
Lemmas \ref{l:Op'F=F:1}, and \ref{l:Op'F=F:2}, and \ref{p:p'-index}. We 
first look at the case $p=2$. 

\begin{Prop} \label{p:2'-index}
Let $G$ be a finite simple group of Lie type in odd characteristic, fix 
$S\in\syl2G$, and set $\calf=\calf_S(G)$. Assume $S\nnsg\calf$. Then 
$O^{2'}(\calf)=\calf$, and $\calf$ is simple.
\end{Prop}

\begin{proof} We will prove that $O^{p'}(\calf)=\calf$. Once we 
know this, then since no proper nontrivial subgroup of $S$ is strongly 
closed with respect to $G$ by \cite[Theorem 1.1]{FF}, it follows by 
Lemma \ref{l:no.str.cl.} that $\calf$ is simple.

By \cite[Proposition 6.2]{BMO2}, we can assume that $G$ 
satisfies Case (III.1) of \cite[Hypotheses 5.1]{BMO2}. Hence $G\cong\gg(q)$ 
for $q\equiv1$ (mod $4$) (i.e., a Chevalley group), or $G\cong{}^2\gg(q)$ 
for $\gg=A_n$, $D_n$, or $E_6$ and $q\equiv1$ (mod $4$). So by \cite[Lemma 
5.3]{BMO2}, there is a centric abelian subgroup $A\nsg S$ such that 
$\Aut_G(A)$ is generated by reflections; in particular, such that 
$O^{2'}(\Aut_G(A))=\Aut_G(A)$. (These results in \cite{BMO2} are stated 
for the groups of universal type, but they carry over easily to the adjoint 
case.) By \cite[Proposition 5.13]{BMO2}, in all cases except when 
$G\cong\Sp_{2n}(q)$ for some $n\ge2$ and $q\equiv5$ (mod $8$), $A$ is 
weakly closed in $\calf$, and hence $\quotfus\calf=1$ and 
$O^{2'}(\calf)=\calf$ by Lemma \ref{theta-F-A}(b). 

Now assume that $G\cong\PSp_{2n}(q)$, where $n\ge2$ and $q\equiv5$ (mod 
$8$), and set $\til{G}=\Sp_{2n}(q)$. More generally, we use tildes for 
subgroups and elements of $\til{G}$, and the same symbol without tilde for 
their images in $G$. Thus $\til{S}\in\syl2{\til{G}}$ and also 
$\til\calf=\calf_{\til{S}}(\til{G})$. Let $(V,\bb)$ be the symplectic space 
on which $\til{G}$ acts as the group of isometries, choose an 
orthogonal decomposition $V=V_1\perp\cdots\perp V_n$ where 
$\dim_{\F_q}(V_i)=2$ for each $i$, and let $\til{z}_i\in\til{G}$ act via 
$-\Id$ on $V_i$ and the identity on the other summands. Set 
$\til{Z}=\gen{\til{z}_1,\dots,\til{z}_n}$: then $C_{\til{G}}(\til{Z})$ is a 
product of $n$ copies of $\Sp_2(q)\cong\SL_2(q)$ and 
$N_{\til{G}}(\til{Z})\cong\Sp_2(q)\wr\Sigma_n$. We can assume that 
$S\in\syl2{G}$ was chosen so that $\til{S}\le N_{\til{G}}(\til{Z})$ (hence 
so that $\til{Z}\le\til{S}$). 
 
Set $\til{B}=C_{\til{S}}(\til{Z})\cong(Q_8)^n$. Each 
$g\in\til{z}_1^{\til\calf}$ acts on $V$ with 2-dimensional $(-1)$-eigenspace, 
hence centralizes the $z_i$ or permutes them via a 2-cycle, and no element 
that permutes them via a 2-cycle can be a square in $\til{S}$. So 
each $\varphi\in\Hom_{\til\calf}(\til{B},\til{S})$ sends $\til{Z}$ to 
itself, and thus sends $\til{B}=C_{\til{S}}(\til{Z})$ to itself. This 
proves that $\til{B}$ is weakly closed in $\til\calf$.

Set $\til{W}=\gen{\til{z}_1\til{z}_2}$. We can assume $\til{S}$ (or the ordering 
of the $\til{z}_i$) was chosen so that $\til{W}$ is fully centralized in 
$\til\calf$. Set $\til\cale=C_{\til\calf}(\til{W})$: the fusion system of 
$C_{\til{G}}(\til{W})\cong\Sp_4(q)\times\Sp_{2n-4}(q)$. By Lemma 
\ref{l:F1xF2}, $\til\cale=\til\cale_1\times\til\cale_2$
and $O^{2'}(\til\cale)=O^{2'}(\til\cale_1)\times O^{2'}(\til\cale_2)$, where 
$\til\cale_1$ and $\til\cale_2$ are the fusion systems of $\Sp_4(q)$ and 
$\Sp_{2n-4}(q)$, respectively. 
The fusion system of $\PSp_4(q)$ is simple by \cite[Proposition 
5.1]{O-rk4}, so $O^{2'}(\til\cale_1)=\til\cale_1$ by Lemma \ref{mod-Z(F)}, and 
hence $\Aut_{O^{2'}(\til\cale)}(\til{B})$ contains $(C_3\wr\Sigma_2)\times1$ 
as a subgroup of $\Aut_{\til\calf}(\til{B})\cong C_3\wr\Sigma_n$. Thus 
$\Aut_{O^{2'}(\til\cale)}(\til{B})$ and $O^{2'}(\Aut_{\til\calf}(\til{B}))$ 
generate $\Aut_{\til\calf}(\til{B})$, so $\quotfus{\til\calf}=1$ by Lemma 
\ref{l:Op'F=F:2}, and hence $\quotfus\calf=1$ by Lemma \ref{mod-Z(F)} again.
\end{proof}

When handling the cases where $p$ is odd, the following list of 
isomorphisms between fusion systems of groups of Lie type will be helpful. 
It is convenient to write $G\sim_pH$ to mean that $G$ and $H$ are finite 
groups whose $p$-fusion systems are isomorphic.

For each prime $p$ and each $q$ prime to $p$, $\ord_p(q)$ denotes 
the multiplicative order of $q$ modulo $p$. 

\begin{Lem} \label{equiv-simp-1}
Fix an odd prime $p$ and a prime power $q$ such that $p\nmid q$. Then 
each of the following inclusions of finite groups 
induces an isomorphism of $p$-fusion systems: 
\begin{enuma} 

\item $\GL_n(q) \le \GL_{n+1}(q)$ when $n\ge2$ and 
$\ord_p(q)\nmid(n+1)$;

\item $\Sp_{2n}(q) \le \GL_{2n}(q)$ (all $n\ge1$) when $\ord_p(q)$ is even;

\item $\SO_{2n+1}(q) \le \GL_{2n+1}(q)$ (all $n\ge1$) when 
$\ord_p(q)$ is even; 

\item $\GO_{2n}^\gee(q) \le \GL_{2n}(q)$ when $\gee=\pm1$, $n\ge1$, 
$\ord_p(q)$ is even, and $q^n\not\equiv-\gee$ (mod $p$); and 

\item $\SO_{2n-1}(q) \le \SO_{2n}^\gee(q)$ when $\gee=\pm1$, $n\ge2$, 
and $q^n\not\equiv\gee$ (mod $p$).

\end{enuma}
\end{Lem}

\begin{proof} By \cite[Lemma A.2]{BMO1}, for each of the inclusions 
(a,b,c,d,e), the $p$-fusion system of the subgroup is a full subcategory of 
the $p$-fusion system of the larger group. So it remains only to check that 
each subgroup has index prime to $p$ in the larger group, and this holds 
since by the well known formulas for the orders of these groups, 
	\begin{align*} 
	v_p\bigl(|\GL_{n+1}(q):\GL_{n}(q)|\bigr) &= v_p(q^{n+1}-1) \\ 
	v_p\bigl(|\GL_{2n}(q):\Sp_{2n}(q)|\bigr) &= 
	\textstyle\sum_{i=1}^nv_p(q^{2i-1}-1) \\ 
	v_p\bigl(|\GL_{2n+1}(q):\SO_{2n+1}(q)|\bigr) &= 
	\textstyle\sum_{i=1}^{n+1}v_p(q^{2i-1}-1) \\ 
	v_p\bigl(|\GL_{2n}(q):\GO_{2n}^\gee(q)|\bigr) &= 
	v_p(q^n+\gee)\cdot\textstyle\sum_{i=1}^{n}v_p(q^{2i-1}-1) \\
	v_p\bigl(|\SO_{2n}^\gee(q):\SO_{2n-1}(q)|\bigr) &= v_p(q^n-\gee) 
	\end{align*} 
(see, e.g., \cite[pp. 19, 70, 141]{Taylor}). 
\end{proof}

\begin{Lem} \label{equiv-simp-2}
Fix an odd prime $p$ and a prime power $q$ such that $p\nmid q$. Then 
there exists a prime power $q^\vee$ such that $\ord_p(q^\vee)=\ord_p(-q)$ 
and $v_p((q^\vee)^{p-1}-1)=v_p(q^{p-1}-1)$. For each such prime power 
$q^\vee$:
\begin{enuma} 

\item $\PSU_n(q)\sim_p \PSL_n(q^\vee)$ for all $n$;
\item $\varOmega_{2n}^\gee(q) \sim_p \varOmega_{2n}^{(-1)^n\gee}(q^\vee)$ for all 
$\gee\in\{\pm1\}$ and all $n$; 
\item $E_6(q) \sim_p \lie2E6(q^\vee)\sim_p F_4(q^\vee)$; and 
\item $\gg(q) \sim_p \gg(q^\vee)$ when $\gg=\varOmega_{2n+1}$ or $\Sp_{2n}$ 
for $n\ge2$, or $\gg=G_2$, $F_4$, $E_7$, or $E_8$. 

\end{enuma}
\end{Lem}

\begin{proof} Set $\ell=v_p(q^{p-1}-1)$. Since 
$(\Z/p^{\ell+1})^\times$ is cyclic, there is $x\in\Z$ whose class modulo 
$p^{\ell+1}$ has order $p\cdot\ord_p(-q)$. By Dirichlet's theorem on primes 
in arithmetic sequences (see \cite[p. 74, Corollary]{Serre}) and since 
$\gcd(x,p)=1$, there are infinitely many primes $q^\vee$ such that 
$q^\vee\equiv x$ (mod $p^{\ell+1}$). For each such $q^\vee$, we have 
$\ord_p(q^\vee)=\ord_p(-q)$ and $v_p((q^\vee)^{p-1}-1)=v_p(q^{p-1}-1)$.

The above relations between groups are all special cases of \cite[Theorem 
A(c,d)]{BMO1}, together with \cite[Example 4.4(b)]{BMO1} (for the second 
equivalence in (c)). Note also the remark in \cite[p. 11]{BMO1} that 
explains why the above conditions on $q^\vee$ and $-q$ hold if and only if 
they generate the same closed subgroup of $\Z_p^\times$ (the $p$-adic 
units).
\end{proof}

We now look at the classical groups when $p$ is odd. We use the following, 
standard notation to describe certain automizers. For a given prime $p$, 
$k\mid m\mid(p-1)$ and $n\ge1$, and a commutative ring $R$ whose group of 
$m$-th roots of unity is cyclic of order $m$, define $G(m,k,n)\le\GL_n(R)$ as 
follows: 
	\[ G(m,k,n) = \bigl\{ \diag(u_1,\dots,u_n) \in\GL_n(R) 
	\,\big|\, \textup{$u_i^m=1$ $\forall\,i$, $(u_1\cdots u_n)^{m/k}=1$} 
	\bigr\} \cdot \Perm(n) \]
where $\Perm(n)\cong\Sigma_n$ is the group of permutation matrices: the 
matrices that permute the canonical basis. Thus, for example, 
$G(m,1,n)\cong C_m\wr\Sigma_n$, and 
$G(m,m,n)\cong(C_m)^{n-1}\rtimes\Sigma_n$. This notation 
will be used in particular when $R=\Z/p^\ell$ (some $\ell\ge1$) and $m\mid(p-1)$.

\begin{Prop} \label{p:p'-ind-1}
Fix an odd prime $p$, and let $G$ be a finite simple linear, unitary, 
symplectic, or orthogonal group in 
defining characteristic different from $p$. Fix $S\in\sylp{G}$, assume $S$ 
is nonabelian, and set $\calf=\calf_S(G)$. Then $O^{p'}(\calf)$ is simple, 
and $\quotfus\calf\cong\autf(S)/\Aut_{O^{p'}(\calf)}(S)$ is as follows:
\begin{enuma} 

\item If $G\cong\PSL^\gee_n(q)$ for some $\gee\in\{\pm1\}$ and $n\ge2$, 
then $\quotfus\calf$ is cyclic of order $\ord_p(\gee q)$. 

\item If $G\cong\PSp_{2n}(q)$ or $\varOmega_{2n+1}(q)$, then 
$\quotfus\calf$ is cyclic of order $\lcm(2,\ord_p(q))$.

\item If $G\cong\POmega_{2n}^\gee(q)$ where $\gee=\pm1$ and $q^n\not\equiv\gee$ 
(mod $p$), then $\quotfus\calf$ is cyclic of order 
$\lcm(2,\ord_p(q))$.

\item If $G\cong\POmega_{2n}^\gee(q)$ where 
$\gee=\pm1$ and $q^n\equiv\gee$ (mod $p$), then 
$\quotfus\calf$ is cyclic of order $\frac12\cdot\lcm(2,\ord_p(q))$. 
If, in addition, $\ord_p(q)$ is even, then $\calf$ is isomorphic to a 
normal subsystem of index $2$ in the $p$-fusion system of $\SL_n(q)$.

\end{enuma}
In each of the above cases, there is an abelian subgroup $A\nsg S$ that is 
$\calf$-centric and weakly closed in $\calf$. More precisely, if we define 
$m$, $\mu$, $\theta$, and $\kappa$ as in Table \ref{tbl:m-mu-theta} and set 
$\ell=v_p(q^\mu-\theta)$, then $A$, $\autf(A)$, and 
$\Aut_{O^{p'}(\calf)}(A)$ are as described in Table \ref{tbl:m-mu-theta}.
	\begin{table}[ht]
	\[ \renewcommand{\arraystretch}{1.3} \setlength{\arraycolsep}{3mm}
	\begin{array}{c|ccccccc}
	\textup{case} & m & \mu & \theta & \kappa & \Aut_\calf(A) & 
	\Aut_{O^{p'}(\calf)}(A) \\\hline
	\textup{(a)} & \ord_p(\gee q) & m & \gee & [n/\mu] & 
	C_{\mu}\wr\Sigma_\kappa & G(\mu,\mu,\kappa) \\
	\textup{(b)} & \ord_p(q) & \lcm(2,m) 
	& (-1)^{m+1} & [2n/\mu] 
	& C_{\mu}\wr\Sigma_\kappa & G(\mu,\mu,\kappa) \\
	\textup{(c)} & \ord_p(q) & \lcm(2,m) 
	& (-1)^{m+1} & [2(n-1)/\mu] 
	& C_{\mu}\wr\Sigma_\kappa & G(\mu,\mu,\kappa) \\
	\textup{(d)} & \ord_p(q) & 
	\lcm(2,m) & (-1)^{m+1} & [2n/\mu] & 
	G(\mu,2,\kappa) & G(\mu,\mu,\kappa) 
	\end{array} \]
	\caption{In all cases, $\kappa\ge p$ and $A$ has exponent 
	$p^\ell$. In all cases except (a) when $p\mid(q-\gee)$, 
	$A\cong(C_{p^\ell})^\kappa$.} 
	\label{tbl:m-mu-theta}
	\end{table}
\end{Prop}

\begin{proof} We first show that $O^{p'}(\calf)$ is simple, assuming the 
other claims hold. If $G=\PSL_n^\gee(q)$ where $\gee=\pm1$ and 
$p\mid(q-\gee)$, then no proper nontrivial subgroup of $S$ is strongly 
closed in $\calf$ by \cite[Theorem 1.2]{FF} and since $S$ is nonabelian. So 
once we know that $O^{p'}(\calf)=\calf$, then Lemma \ref{l:no.str.cl.} 
implies that $\calf$ is simple.

In all other cases, by Table \ref{tbl:m-mu-theta}, $\Omega_1(A)$ is a 
simple $\F_p\Aut_{O^{p'}(\calf)}(A)$-module and $|S/A|<|A|$. No proper 
nontrivial subgroup of $S$ is strongly closed in $\calf$ by \cite[Theorem 
1.2]{FF} again, so we are in the situation of Lemma~\ref{l:no.str.cl.}. 
Assume $O^{p'}(\calf)$ is not simple: then there are subgroups 
$T_1,\dots,T_k$ (some $k\ge 2$) that are strongly closed in $O^{p'}(\calf)$ 
and $\Aut_\calf(S)$-conjugate to each other, and such that 
$S=T_1\times\cdots\times T_k$. For each $i$, let $\pi_i\:S\too T_i$ be the 
projection. Then $\Ker(\pi_i)\cap A$ is normalized by 
$\Aut_{O^{p'}(\calf)}(A)$, and since $\Omega_1(A)$ is a simple module, 
either $\Ker(\pi_i)\cap A=1$ or $\Omega_1(A)\le\Ker(\pi_i)$. If 
$\Omega_1(A)\le\Ker(\pi_i)$ for one index $i$ then it holds for all $i$, 
which would imply $A=1$. Thus $\Ker(\pi_i)\cap A=1$ for each $i$, so 
$\pi_i$ sends $A$ injectively into $T_i$, and 
$|S|=\prod_{i=1}^k|T_i|\ge|A|^k\ge|A|^2$, contradicting the earlier 
observation that $|S/A|<|A|$. Thus $O^{p'}(\calf)$ is simple.

Note also that $\kappa\ge p$ in all cases, assuming the other claims 
hold: $p\mid|\autf(A)|$ since $A\nsg S$ and $S$ is nonabelian, and 
$\autf(A)$ has index $1$ or $2$ in $C_\mu\wr\Sigma_\kappa$ where 
$p\nmid\mu$.

\smallskip

\noindent\textbf{(a) } By Lemma \ref{equiv-simp-2}(a), it suffices to prove 
this when $G\cong\PSL_n(q)$: the case handled by Ruiz in \cite{Ruiz}. We 
sketch a slightly different argument, based on Lemmas \ref{l:Op'F=F:2} and 
\ref{p:p'-index}. Recall that $m=\mu=\ord_p(q)$.

\noindent\textbf{Case (a.1): } Assume $m=1$; i.e., $p\mid(q-1)$. Set 
$\4G=\SL_n(q)$, and let $A\le\4G$ be the subgroup of diagonal matrices of 
$p$-power order and determinant 1. Thus $A\cong(C_{p^\ell})^{n-1}$ 
(recall that $\ell=v_p(q-1)$). Choose $\4S\in\sylp{\4G}$ such that 
$A\le\4S\le N_{\4G}(A)$, and set $\4\calf=\calf_{\4S}(\4G)$. Then 
\cite[Hypotheses 5.1 case (III.1)]{BMO2} is satisfied, and so $A$ is weakly 
closed in $\4S$ by \cite[Proposition 5.13]{BMO2}. 

By Lemma \ref{mod-Z(F)}, $\quotfus\calf\cong\quotfus{\4\calf}$, and hence 
it suffices to show that $\quotfus{\4\calf}=1$. Set 
	\[ \Gamma= \Aut_{\4\calf}(A)\cong\Sigma_n \qquad\textup{and}\qquad 
	\Gamma_0 = \Aut_{O^{p'}(\4\calf)}(A). \] 
If $\quotfus{\4\calf}\ne1$, then $\Gamma_0<\Gamma$ is normal of index 
prime to $p$ by Lemma \ref{theta-F-A}(a). Also, $n\ge p$ since $S$ is 
nonabelian, and since $p$ is odd, this implies $\Gamma_0\cong A_n$ (the 
alternating group).


If $p\le n\le2p-1$, then $|S/A|=p$, and $|[A,S]|\ge p^2$. (Note that if 
$p=n=3$, then $\ell>1$ since $S\in\sylp{\PSL_n(q)}$ is assumed to be 
nonabelian.) If $\quotfus{\4\calf}\ne1$, then $\Gamma_0\cong A_p$, so that 
$|\Aut_{\Gamma_0}(U)|=(p-1)/2$ for $U\in\sylp{\Gamma_0}$. Then $A\nsg 
O^{p'}(\4\calf)$ by Lemma \ref{p:p-1}, and hence $A\nsg\4\calf$. But 
this is impossible: $A$ cannot be strongly closed since each element of 
$\4G$ of $p$-power order is diagonalizable. So $\quotfus{\4\calf}=1$.

Now assume $n\ge2p$. Let $A_0\le A$ be the subgroup of all diagonal 
matrices $\diag(a_1,\dots,a_n)$ of $p$-power order such that 
$a_1=a_2=\dots=a_{p+1}$. Thus $C_{\4G}(A_0)\cong\4G\cap(\GL_{p+1}(q)\times 
(\F_q^\times)^{n-p-1})$. Set 
$G_0=O^{p'}(O^p(C_{\4G}(A_0)))\cong\SL_{p+1}(q)$, so that $G_0\nsg 
C_{\4G}(A_0)$ and $C_{\4G}(A_0)/G_0\cong(\F_q^\times)^{n-p-1}$. 

Set $S_0=G_0\cap C_{\4S}(A_0)\in\sylp{G_0}$ and $\calf_0=\calf_{S_0}(G_0)$. 
By Lemma \ref{diag-aut}, $\calf_0$ is normal of $p$-power index in 
$C_{\4\calf}(A_0)$. So by Lemmas \ref{l:Op'Op<Op'} and \ref{l:Op'F=F:1}(b), 
	\[ O^{p'}(\calf_0)\le O^{p'}(C_{\4\calf}(A_0))\le O^{p'}(\4\calf). 
	\]
We already saw that $\Aut_{O^{p'}(\calf_0)}(A)\cong\Sigma_{p+1}$, so 
$\Aut_{\4\calf}(A)=O^{p'}(\Aut_{\4\calf}(A))\cdot\Aut_{O^{p'}(\calf_0)}(A)$, 
and hence $\quotfus{\4\calf}=1$ by Lemma \ref{l:Op'F=F:2}.

\noindent\textbf{Case (a.2): } Now assume $m>1$; i.e., that 
$p\nmid(q-1)$. Fix a vector space $V\cong(\F_q)^n$, and set 
$\til{G}=\Aut_{\F_q}(V)\cong\GL_n(q)$ and 
$\til{G}_0=O^{p'}(\til{G})\cong\SL_n(q)$. Note that $\calf=\calf_S(G)$ is 
isomorphic to the fusion system of $\til{G}_0$ since 
$p\nmid|Z(\til{G}_0)|$, 
and we will show later that it is isomorphic to that of $\til{G}$. 

Fix an irreducible polynomial $f\mid(X^p-1)$ of degree $m$ in $\F_q[X]$, and 
identify $\F_{q^m}\cong\F_q[X]/(f)$. Choose a decomposition 
$V=V_1\oplus\dots\oplus V_\kappa\oplus W$, where $\dim(V_i)=m$ for each $i$ 
and $\dim(W)=n-m\kappa<m$ (recall $\kappa=[n/m]$). Thus 
$p\nmid|\Aut_{\F_q}(W)|$. Set $\5V=V_1\oplus\dots\oplus V_\kappa$. Choose 
$z_i\in\Aut_{\F_q}(V_i)$ with characteristic polynomial $f$, and set 
$z=z_1\oplus\cdots\oplus z_\kappa\oplus\Id_W\in C_{\til{G}}(W)$. Thus $z$ 
determines an $\F_{q^m}$-vector space structure on $\5V$, and we set 
	\begin{align*} 
	H&=C_{\til{G}}(z)=\Aut_{\F_{q^m}}(\5V)\times\Aut_{\F_q}(W) \cong 
	\GL_\kappa(q^m)\times\GL_{n-m\kappa}(q) \\
	H_0&=O^{p'}(H)\cong\SL_\kappa(q^m)\rtimes C_{p^\ell}. 
	\end{align*} 

Now set 
	\[ A = O_p\bigl(\Aut_{\F_{q^m}}(V_1)\times\dots\times
	\Aut_{\F_{q^m}}(V_\kappa) \bigr) \cong 
	O_p\bigl((\F_{q^m}^\times)^\kappa\bigr) \cong 
	(C_{p^\ell})^\kappa. \] 
Thus $\Omega_1(A)=\gen{z_1,\dots,z_\kappa}$. Each element of 
$N_{\til{G}}(A)$ permutes the subgroups $\gen{z_i}$. So 
	\beqq N_{\til{G}}(A) \cong (\F_{q^m}^\times \rtimes C_m)\wr\Sigma_\kappa 
	\times \GL_{n-m\kappa}(q) 
	\quad\textup{and}\quad 
	C_{\til{G}}(A) \cong (\F_{q^m}^\times)^\kappa \times 
	\GL_{n-m\kappa}(q), \label{e:CG(A)} \eeqq
and hence 
	\[ \Aut_{\til{G}}(A)=\Aut_{\til{G}_0}(A)\cong C_m\wr\Sigma_\kappa 
	\quad\textup{and}\quad 
	\Aut_H(A)\cong \Sigma_\kappa. \]
 Choose $S_0\in\sylp{\Sigma_\kappa}$, and set $S=AS_0$. 
Then $S\in\sylp{N_{\til{G}}(A)}$, and from the formula for 
$|\til{G}|=|\GL_n(q)|$, we see that 
	\[ v_p(|\til{G}|) = \sum_{i=1}^nv_p(q^i-1) = \sum_{i=1}^\kappa 
	v_p(q^{mi}-1) = 
	\sum_{i=1}^\kappa\bigl(v_p(q^m-1)+v_p(i)\bigr) = 
	\kappa\ell+v_p(\kappa!) = v_p(|S|) \]
and hence that $S\in\sylp{\til{G}}=\sylp{\til{G}_0}$. We can thus identify 
$\calf=\calf_S(\til{G}_0)$. 

Set $X=\{t\in A\,|\,t^\calf\subseteq A\}$. For each $x\in S\sminus A$ of 
order $p$, $x$ acts on the set $\{V_i\}$ via a permutation of order $p$, 
and hence $\dim_{\F_q}(C_V(x))\ge m$. Since $C_V(z)=W$, where $\dim(W)<m$, 
this proves that $z\in X\cap Z(S)$. More generally, $\Omega_1(A)$ is 
generated by elements $a$ such that $C_V(a)=W$, so $\Omega_1(A)$ and 
$A=C_S(\Omega_1(A))$ are both weakly closed in $\calf$.

Consider the surjective homomorphism 
	\[ \thet\calf{A} \: \autf(A) = \Aut_{\til{G}}(A) \Onto5{} 
	\quotfus\calf \]
of Lemma \ref{theta-F-A}. We want to apply Lemma \ref{p:p'-index}(b) with 
$\gen{z}$ in the role of $Z$. Set $\cale=C_\calf(z)=\calf_S(H)$: the 
$p$-fusion system of $C_{\til{G}}(z)$ and hence of $\GL_\kappa(q^m)$. We 
showed in Case (a.1) that $O^{p'}(\cale)$ contains the fusion system of 
$\SL_\kappa(q^m)$, and hence that 
	\[ \Aut_{O^{p'}(\cale)}(A) = \Aut_\cale(A) 
	= \Aut_H(A)\cong\Sigma_\kappa. \]
So in the notation of Lemma \ref{p:p'-index}, $K_0=K\cong 
G(m,m,\kappa)$: the normal closure of $\Aut_\cale(A)$ in $\autf(A)\cong 
C_m\wr\Sigma_\kappa$. Hence $\Ker(\thet\calf{A})=K$ by that lemma, and so 
	\[ \quotfus\calf\cong \autf(A)/K\cong C_m. \]

Let $N\:\F_{q^m}^\times\too\F_q^\times$ be the norm map: $N(u)=u\cdot 
u^q\cdots u^{q^{m-1}}=u^{(q^m-1)/(q-1)}$ for $u\in\F_{q^m}^\times$. Thus 
$N$ is surjective. For $u\in\F_{q^m}^\times$, 
$\det_{\F_q}(\F_{q^m}\xto{~u\cdot~}\F_{q^m})=N(u)$ (see, e.g., 
\cite[Proposition VI.5.6]{Lang}). Together with the description of 
$C_{\til{G}}(A)$ in \eqref{e:CG(A)}, this shows that 
$\til{G}=\til{G}_0C_{\til{G}}(A)$. We just saw that $A$ is weakly closed in 
$\calf$ (hence in $\calf_0$), and Lemma \ref{diag-aut} now implies that 
$\calf_S(\til{G})=\calf_S(\til{G}_0)\cong\calf$.

\smallskip

\noindent\textbf{(b,c,d) } If $m=\ord_p(q)$ is odd, then by Lemma 
\ref{equiv-simp-2}(b,d), there is a prime power $q^\vee$ such that 
$\ord_p(q^\vee)=\ord_p(-q)=2m$, and for each $n$, 
$\Sp_{2n}(q)\sim_p\Sp_{2n}(q^\vee)$, 
$\varOmega_{2n+1}(q)\sim_p\varOmega_{2n+1}(q^\vee)$, and 
$\varOmega_{2n}^\gee(q)\sim_p\varOmega_{2n}^{(-1)^n\gee}(q^\vee)$ (for 
$\gee=\pm1$). Also, 
	\[ q^n\equiv\gee \pmod{p} ~\iff~ \textup{$\gee=1$ and $m\mid n$} 
	~\iff~ (q^\vee)^n\equiv(-1)^n=(-1)^n\gee \pmod{p}, \]
and hence the claims in (b), (c), and (d) and in Table \ref{tbl:m-mu-theta} all 
hold when $\ord_p(q)$ is odd if they hold when $\ord_p(q)$ is even.

\smallskip

\noindent\textbf{(b) } Assume $\ord_p(q)$ is even. If $G\cong\PSp_{2n}(q)$, 
then $G\sim_p\Sp_{2n}(q)\sim_p\GL_{2n}(q)\sim_p\SL_{2n}(q)$ by Lemmas 
\ref{equiv-simp-1}(b) and \ref{diag-aut}, and hence $\quotfus\calf$ is 
cyclic of order $\ord_p(q)$ by (a). If $G\cong\varOmega_{2n+1}(q)$, then 
$G\sim_p\SO_{2n+1}(q)\sim_p\GL_{2n+1}(q)\sim_p\SL_{2n+1}(q)$ by Lemmas 
\ref{equiv-simp-1}(c) and \ref{diag-aut}, and hence $\quotfus\calf$ is 
cyclic of order $\ord_p(q)$. In both cases, the claims in Table 
\ref{tbl:m-mu-theta} about $A$, $\autf(A)$, and $\Aut_{O^{p'}(\calf)}(A)$ 
all follow from (a). 

\smallskip

\noindent\textbf{(c) } Assume $G\cong\POmega_{2n}^\gee(q)$, where 
$\gee=\pm1$ and $q^n\not\equiv\gee$ (mod $p$). 
Then $G\sim_p\SO_{2n}^\gee(q)\sim_p\SO_{2n-1}(q)$ by Lemmas \ref{diag-aut} 
and \ref{equiv-simp-1}(e). So by (b), $\quotfus\calf$ is cyclic of order 
$\lcm(2,\ord_p(q))$, and the claims in Table \ref{tbl:m-mu-theta} about 
$A$, $\autf(A)$, and $\Aut_{O^{p'}(\calf)}(A)$ all hold. 

\smallskip

\noindent\textbf{(d) } Assume $G\cong\POmega_{2n}^\gee(q)$, where 
$\gee=\pm1$ and $q^n\equiv\gee$ (mod $p$), and where $m=\ord_p(q)$ is even. 
Set $\4G=\GO_{2n}^\gee(q)$ and let $\4\calf$ be the fusion system of $\4G$; 
thus $\4\calf$ contains $\calf$ as a normal subsystem of index $1$ or $2$. 
Then $\4G\sim_p \GL_{2n}(q)$ by Lemma \ref{equiv-simp-1}(d), and 
$\GL_{2n}(q)\sim_p \SL_{2n}(q)$ as shown in the proof of case (a.2). So 
$\quotfus{\4\calf}$ is cyclic of order $m$, and $\quotfus\calf$ is cyclic 
of order $m/2$ or $m$.

Recall that $\mu=\lcm(2,m)=m$ and $\kappa=[2n/\mu]$. Set 
$\lambda=m/2=\mu/2$, so that $\kappa=[n/\lambda]$. Since 
$q^n\equiv\gee=\pm1$ by assumption, $\lambda\mid n$, and so 
$n=\kappa\lambda$. Since $\ord_p(q)=m=2\lambda$, $p$ divides 
$q^m-1=(q^\lambda-1)(q^\lambda+1)$, but $p$ does not divide $q^\lambda-1$. 
Hence $p\mid(q^\lambda+1)$.

Set $V=(\F_{q^m})^\kappa$ as an $\F_q$-vector space. We want to describe 
$\4G$ as an orthogonal group on $V$, and must first find the right 
quadratic form on this space. Let $(x\mapsto\4x=x^{q^\lambda})$ be the 
field automorphism of order $2$ in $\F_{q^m}$, and consider the composite 
	\[ \qq_0 \: \F_{q^m} \Right3{N} \F_{q^\lambda} \Right3{\Tr} \F_q, 
	\quad (N(x)=x\4x=x^{q^\lambda+1},\quad 
	\Tr(x)=x+x^q+\dots+x^{q^{\lambda-1}}). \]
Thus $N$ and $\Tr$ are the norm and trace maps for the respective field 
extensions. Then $\qq_0$ is quadratic, has associated bilinear form 
	$\bb(x,y) = \Tr(x\4y+\4xy) = \Tr_{\F_{q^m}/\F_q}(x\4y)$, 
and hence is nondegenerate. Since $(x\mapsto ux)\in\GO(\F_{q^m},\qq_0)$ for 
all $u\in N^{-1}(1)$, the orthogonal group contains a cyclic subgroup of order 
$p\mid(q^\lambda+1)$. Hence $\GO(\F_{q^m},\qq_0)\cong\GO_{2\lambda}^-(q)$, 
since $p\nmid|\GO_{2\lambda}^+(q)|$ (see, e.g., \cite[p. 141]{Taylor}). So 
$\qq_0$ is a quadratic form of type $-1$. 

Now define $\qq\:V=(\F_{q^m})^\kappa\too\F_q$ by setting 
$ \qq(x_1,\dots,x_\kappa) = \textstyle\sum_{i=1}^\kappa\qq_0(x_i)$. 
Then $\qq$ is nondegenerate of type $(-1)^\kappa$, and 
$(-1)^\kappa\equiv(q^\lambda)^\kappa=q^n\equiv\gee$ 
(mod $p$). So $\GO(V,\qq)\cong\GO_{2n}^\gee(q)=\4G$.

We next claim, for $u_1,\dots,u_\kappa\in\F_{q^m}^\times$, that 
	\beqq \textup{$\diag(u_1,\dots,u_\kappa)$ is orthogonal $\iff$ 
	$u_i\4u_i=1$ for each $i$.} \label{e:N(u_i)=1} \eeqq
This means showing, for each $u\in\F_{q^m}^\times$, that 
$\qq_0(ux)=\qq_0(x)$ for all $x\in\F_{q^m}$, or equivalently 
$\Tr(u\4ux\4x)=\Tr(x\4x)$, if and only if $u\4u=1$. Since $u\mapsto 
u\4u=u^{1+q^\lambda}$ sends $\F_{q^m}=\F_{q^{2\lambda}}$ surjectively onto 
$\F_{q^\lambda}$, we must prove, for all $v\in\F_{q^\lambda}^\times$, that 
$\Tr((v-1)\F_{q^\lambda})=0$ only if $v=1$, and this is clear since $v\ne1$ 
implies $(v-1)\F_{q^\lambda}=\F_{q^\lambda}$. 

Thus the subgroup 
	\[ A = \bigl\{ \diag(u_1,\dots,u_\kappa) \,\big|\, u_i^{p^\ell}=1 
	\textup{ for all $i$} \bigr\} \cong (C_{p^\ell})^\kappa \]
is contained in $\GO(V,\qq)$ (recall $\ell=v_p(q^m-1)=v_p(q^\lambda+1)$). 
We saw in the proof of (a) that $N_{\GL(V)}(A)\cong(\F_{q^m}^\times\rtimes 
C_m)\wr\Sigma_\kappa$. A diagonal matrix is orthogonal only if it has 
order dividing $q^\lambda+1$, and the field automorphisms and permutations 
are clearly orthogonal. So 
	\[ N_{\4G}(A) = N_{\GO(V,\qq)}(A) \cong (C_{q^\lambda+1}\rtimes C_m) 
	\wr \Sigma_\kappa \]
and hence $\Aut_{\4G}(A) \cong C_m\wr\Sigma_\kappa$
as expected (recall $\4G\sim_p\GL_{2n}(q)\sim_p\SL_{2n}(q)$).

If $q$ is odd, then the diagonal matrices of order dividing $q^\lambda+1$ 
have determinant $1$ (over $\F_q$), and the permutations of coordinates 
in $V\cong(\F_{q^m})^\kappa$ have determinant $1$ since $m$ is even. 
Define $\alpha\in\GO(V,\qq)$ by setting 
$\alpha(x_1,\dots,x_\kappa)=(x_1^q,x_2,\dots,x_\kappa)$. By the normal 
basis theorem (see \cite[\S\,4.14]{Jacobson}), the field automorphism 
$(x\mapsto x^q)$ permutes cyclically an $\F_q$-basis for $\F_{q^m}$, and 
has determinant $-1$ since $2\mid m$. So $\det(\alpha)=-1$. It follows 
that $\Aut_G(A)\cong G(\mu,2,\kappa)$ has index $2$ in $\Aut_{\4G}(A)$, and 
hence that $\quotfus\calf$ is cyclic of order $\lambda=\mu/2$. 

If $q$ is a power of $2$, then a similar argument applies upon replacing 
the determinant by the Dickson invariant $D(\alpha)=\dim_{\F_q}([V,\alpha])$ (mod 
$2$) for $\alpha\in\GO(V,\qq)$ (see, e.g., \cite[Theorem 11.43]{Taylor}). 
The diagonal matrices and the permutations all lie in 
$\Ker(D)=\varOmega(V,\qq)$ since $\dim([V,\alpha])\in m\Z\le2\Z$ for all 
such $\alpha$. If $\alpha$ sends $(x_1,\dots,x_\kappa)$ to 
$(x_1^q,x_2,\dots,x_\kappa)$, then $[V,\alpha]$ has codimension $1$, 
as an $\F_q$-vector space, in 
$\F_{q^m}\times0$ by the normal basis theorem again, and hence 
$\dim_{\F_q}([V,\alpha])=m-1$ is odd.
\end{proof}

It remains to consider the exceptional groups of Lie type. Note that 
if $G$ is of exceptional Lie type and $S\in\sylp{G}$, then 
either $p$ divides the order of the Weyl group of the associated 
algebraic group, or else $S$ is abelian. In particular, $S$ can be 
nonabelian only for primes $p\le7$.

\begin{Prop} \label{p:p'-ind-2}
Fix an odd prime $p$, and let $G$ be a finite simple group of exceptional 
Lie type in defining characteristic different from $p$. Fix $S\in\sylp{G}$, 
assume $S\nnsg\calf$, and set $\calf=\calf_S(G)$. Then either 
\begin{enuma} 

\item $p=3$ and $G\cong\lie3D4(q)$ where $3\nmid q$, $\lie2F4(q)$ where 
$q=2^{2k+1}$ for $k\ge1$, or $\lie2F4(2)'$; or

\item $G\cong F_4(q)$, $E_6(q)$, $\lie2E6(q)$, $E_7(q)$, 
or $E_8(q)$, where $q\equiv\pm1$ (mod $p$) and $(G,p)$ is one of the pairs 
listed in Table \ref{tbl:AutG(A)-except}; or 

\item $p=5$ and $G\cong E_8(q)$, where $q\equiv\pm2$ (mod $5$); or

\item $p=3$, $G\cong G_2(q)$ where $q\equiv\pm1$ (mod $9$), 
$|O_3(\calf)|=3$, and $O^{3'}(\calf)$ has index $2$ in $\calf$ and is 
realized by $\SL_3^\pm(q)$.

\end{enuma}
In all cases except in case (a) when $G\cong\lie2F4(2)'$ or 
$G\cong\lie2F4(2^{2k+1})$ for $3\nmid(2k+1)$ (and hence 
$S\cong3^{1+2}_+$), there is a unique abelian subgroup $A\nsg S$ of maximal 
rank such that $C_S(A)=A$, and $A$ and $\Aut_G(A)$ are as described 
in Table \ref{tbl:AutG(A)-except}. In all of the cases (a), (b), and (c), 
except in case (a) when $G\cong\lie2F4(2^{2k+1})$ for $3\nmid(2k+1)$, 
$\calf$ is simple, and in particular, $O^{p'}(\calf)=\calf$.
	\begin{table}[ht]
	\[ \renewcommand{\arraystretch}{1.3} 
	\begin{array}{c|ccccc}
	\textup{case} & p & G & q & A & \Aut_\calf(A) \\\hline
	\textup{(a)} & 3 & \lie2F4(q) & q=2^{2n+1},~9\mid(q+1) & C_{3^\ell}\times 
	C_{3^\ell} & \GL_2(3) \\
	\textup{(a)} & 3 & \lie3D4(q) & 3\nmid q & C_{3^{\ell+1}}\times 
	C_{3^\ell} & \Sigma_3\times C_2 \\
	\textup{(b)} & 3 & F_4(q) & q\equiv\pm1~\textup{(mod $p$)}
	& (C_{p^\ell})^4 & W(F_4) \\
	\textup{(b)} & 3 & E_6^\gee(q) & q\equiv\gee~(\textup{mod $p$}) & 
	(C_{3^\ell})^5\times C_{3^{\ell-1}} & W(E_6) \\
	\textup{(b)} & 3 & E_6^\gee(q) & q\equiv-\gee~(\textup{mod $p$}) & 
	(C_{p^\ell})^4 & W(F_4) \\
	\textup{(b)} & 5 & E_6^\gee(q) & q\equiv\gee~(\textup{mod $p$}) & 
	(C_{5^\ell})^6 & W(E_6) \\
	\textup{(b)} & 3,5,7 & E_n(q)~\textup{($n=7,8$)} & 
	q\equiv\pm1~\textup{(mod $p$)} & (C_{p^\ell})^n & W(E_n) \\
	\textup{(c)} & 5 & E_8(q) & q\equiv\pm2~(\textup{mod $p$}) & 
	(C_{p^\ell})^4 & (C_4\circ2^{1+4}).\Sigma_6 \\
	\textup{(d)} & 3 & G_2(q) & q\equiv\pm1~(\textup{mod $9$}) & 
	(C_{3^\ell})^2 & \Sigma_3\times C_2 \\
	\end{array} \]
	\caption{\small{Here, $\ell=v_p(q^2-1)$, except in case (c) where 
	$\ell=v_5(q^2+1)$, and $\gee=\pm1$.}} 
	\label{tbl:AutG(A)-except}
	\end{table}
\end{Prop}

\begin{proof} Assume $p$ is an odd prime and $G$ is a finite simple 
group of exceptional Lie type in defining characteristic different from 
$p$. Choose $S\in\sylp{G}$, and set $\calf=\calf_S(G)$. 
\begin{itemize} 
\item If $G$ is a Suzuki or Ree group or $G\cong\lie3D4(q)$, then either 
$S$ is abelian and hence $S\nsg\calf$, or $(G,p)$ is one of the pairs 
appearing in (a).

\item If $G\cong F_4(q)$, $E_n(q)$, or $\lie2E6(q)$, where $p\nmid q$, then 
by \cite[(10-1(3))]{GL}, either $S$ 
is abelian, or for $m=\ord_p(q)$, some cyclotomic polynomial 
$\Phi_{mp^a}(q)$ for $a\ge1$ appears in the formula for $|G|$. This 
occurs only when $(G,p)$ is one of the pairs listed in Table 
\ref{tbl:AutG(A)-except}, case (b) or (c). (Note that if $p=5$ and 
$G=E_6^\gee(q)$ for $\gee=\pm1$ and $q\equiv-\gee$ (mod $p$), then $S$ is 
abelian by this criterion.)

\item If $G\cong G_2(q)$ where $p\nmid q$, then either $p\ge5$ and $S$ is 
abelian; or $p=3$, $v_3(q^2-1)=1$, and $S\nsg\calf$ by 
\cite[16.11.4]{A-gfit}; or $q\equiv\pm1$ (mod $9$) and $(G,p)$ is as in 
(d).

\end{itemize}
Thus whenever $S\nnsg\calf$, the pair $(G,p)$ is one of those listed 
in cases (a)--(d), and in fact, one of those listed in 
Table \ref{tbl:AutG(A)-except}.

In cases (a), (b), and (c), except in certain cases when $p=3$ and 
$G\cong\lie2F4(q)$, we will prove below that $O^{p'}(\calf)=\calf$. 
Once this has been shown, then $\calf$ is simple by Lemma 
\ref{l:no.str.cl.}, together with \cite[Theorem 1.2]{FF} which says that no 
proper nontrivial subgroup of $S$ is strongly closed in $\calf$. 

\smallskip

\noindent\textbf{(a) } Assume $p=3$. If $G=\lie2F4(2)'$ and 
$S\in\syl3G$, then by \cite[Lemma 4.13]{RV}, $S$ is extraspecial of order 
$3^3$ and exponent $3$, $\Out_G(S)\cong D_8$, and all subgroups of order 
$9$ in $S$ are $\calf$-radical. A normal subsystem of index $2$ would have 
the same radical subgroups, and so there is no such subsystem by the same 
lemma.

If $G=\lie2F4(q)$ where $q=2^{2k+1}$ for $k\ge1$ and $3\nmid(2k+1)$, then 
$\ell=v_3(q^2-1)=1$, and so $S$ is again extraspecial of order $3^3$. In 
this case, the subgroup $\lie2F4(2)$ controls $3$-fusion in $G$, and the 
fusion system of $\lie2F4(2)'$ is normal of index $2$ in $\calf$. This 
follows from \cite[Lemma 4.13]{RV} (where all saturated fusion systems over 
$S$ containing that of $\lie2F4(2)'$ are listed) and the main theorem in 
\cite{Malle} (with its information about the subgroups of $G$). 

If $G=\lie2F4(q)$ where $q=2^{2k+1}$ for $k\ge1$ and $3\mid(2k+1)$, 
then $9|(q+1)$. By the main theorem in \cite{Malle}, $G$ contains a 
maximal torus $H\cong C_{q+1}\times C_{q+1}$ such that 
$\Aut_G(H)\cong\GL_2(3)$ and $N_G(H)\ge S$. Set $A=O_p(H)\cong 
C_{3^\ell}\times C_{3^\ell}$, where $\ell=v_3(q+1)=v_3(q^2-1)\ge2$: then 
$|S/A|=3$, $|[S,A]|=3^\ell\ge3^2$, and $\Aut_G(A)\cong\GL_2(3)$. By Lemma 
\ref{theta-F-A}(a,b), if $\quotfus\calf\ne1$, then it has order $2$ and 
$\Aut_{O^{3'}(\calf)}(A)=O^{3'}(\autf(A))\cong\SL_2(3)\cong2A_4$. But this 
contradicts Lemma \ref{p:p-1}.

If $G=\lie3D4(q)$ where $3\nmid q$, then set $\ell=v_3(q^2-1)$, so that 
$v_3(|G|)=2(\ell+1)$. By the main theorem in \cite{Kleidman-3D4}, $S$ 
contains a unique subgroup $A\cong C_{3^{\ell+1}}\times C_{3^\ell}$ 
of index $3$, and $\Aut_G(A)\cong\Sigma_3\times C_2\cong D_{12}$. (If 
$q\equiv\gee$ (mod $3$) for $\gee\in\{\pm1\}$, then $N_G(A)$ is contained 
inside the maximal subgroup $(C_{q^2+\gee 
q+1}\times_{C_3}\SL_3^\gee(q))\rtimes\Sigma_3$ listed in 
\cite{Kleidman-3D4}.) By the conditions listed in \cite[Theorem 
2.8(a)]{indp1}, there can be no normal subsystem of index $2$ in 
$\calf=\calf_S(G)$. 

Alternatively, the result for $G=\lie2F4(q)$ or $\lie3D4(q)$ follows from 
Tables 2 and 4 in \cite{DRV}: when $m=v_3(|G|)$, then $S\cong B(3,m;0,0,0)$ 
in the notation of \cite{DRV}, and those two tables list all saturated 
fusion systems over these groups.

\smallskip

\noindent\textbf{(b) } If $q\equiv\gee$ (mod $p$) where $\gee=\pm1$, 
then by  Lemma \ref{equiv-simp-2}(c), there is another prime power 
$q^\vee$ such that $\ord_p(q^\vee)=\ord_p(-q)$ (hence $q^\vee\equiv-\gee$ 
(mod $p$)) and $E_6(q)\sim_p \lie2E6(q^\vee)\sim_p F_4(q^\vee)$. This, 
together with similar applications of Lemma \ref{equiv-simp-2}(d), shows 
that it suffices to prove (b) when $G=\gg(q)$ for $q\equiv1$ (mod $p$) and 
$\gg=F_4$ or $E_n$. 

A description of $A$ in these cases can be found in \cite[10-1(2)]{GL} or 
\cite[Theorem 4.10.2(c)]{GLS3}, and more details on $\Aut_G(A)$ are given 
in \cite[Lemmas 6.1 \& 5.3]{BMO2}. In particular, $A$ is homocyclic of rank 
$\rk(\gg)$ and exponent $p^\ell$ (where $\ell=v_p(q-1)$), except when 
$\gg=E_6$ and $p=3$, and $\autf(A)\cong W(\gg)$. Also, $A$ is weakly 
closed in $\calf$ by \cite[Proposition 5.13]{BMO2}. Recall that we are 
working with the simple groups, hence the adjoint forms, which is why the 
description of $A$ is slightly different when $(G,p)=(E_6(q),3)$.

It remains to show that $O^{p'}(\calf)=\calf$, still assuming 
$G=\gg(q)$ for $\gg=F_4$ or $E_n$ and $q\equiv1$ (mod $p$) (and $S$ is 
nonabelian). Since $\autf(A)\cong W(\gg)$ is the usual Weyl group for 
$G$, we have 
	\beqq \autf(A)/O^{p'}(\autf(A)) 
	\cong \begin{cases} 
	C_2\times C_2 & \textup{if $\gg=F_4$ and $p=3$} \\
	C_2 & \textup{if $\gg=E_6,E_7,E_8$ and $3\le 
	p\mid|W(\gg)|$.}
	\end{cases} \label{e:AutF(A)/Op'} \eeqq
When $p=3$, these follow from the ``standard presentation'' of $W(\gg)$ as 
a group of reflections acting on a vector space, with a fundamental system 
of reflections as generating set \cite[Theorem 2.4.1]{Carter}. When $p=5,7$, 
the result follows from the fact that $W(E_n)$ contains a quasisimple 
subgroup of index $2$. In all cases, we set $\ell=v_p(q-1)\ge1$.

We consider two different cases.

\noindent\textbf{Case (b.1) } Assume $|S/A|=p$; equivalently, 
$v_p(|W(\gg)|)=1$. Then either $p=5$ and $\gg=E_6$ or 
$E_7$, or $p=7$ and $\gg=E_7$ or $E_8$, and we are in the situation of 
\cite[Section 2]{indp2}. We refer to Notation 2.4 in that paper, and in 
particular to the subgroups and homomorphism
	\beqq \begin{split} 
	Z &= Z(S) \quad\textup{and}\quad Z_0=Z\cap[S,S] 
	\qquad\textup{($|Z_0|=p$ by \cite[Lemma 2.2(c)]{indp2})} \\
	\Aut_\cale^\vee(S) &= \bigl\{\alpha\in\Aut_\cale(S) \,\big|\, 
	[\alpha,Z]\le Z_0 \bigr\} \qquad\textup{for}~ 
	O^{p'}(\calf) \le\cale\le\calf \\
	\Delta_t &= \{ (r,r^t) \,|\, r\in (\Z/p)^\times \} \le 
	\Delta\defeq (\Z/p)^\times\times(\Z/p)^\times \quad\textup{for 
	$t\in\Z$} \\
	\mu\:&\Aut(S) \Right2{} \Delta \quad\textup{where}\quad 
	\mu(\alpha)=(r,s) ~\textup{if}~ \begin{cases} 
	\alpha(x)\in x^rA & \textup{for $x\in S\sminus A$,} \\ 
	\alpha(g)=g^s & \textup{for $g\in Z_0$.}
	\end{cases} 
	\end{split} \label{e:COS-defs} \eeqq
 By \cite[Lemma 2.6(a)]{indp2} and since $A\nnsg\calf$, 
$\mu(\Aut_{O^{p'}(\calf)}^\vee(S))$ must contain one 
of the subgroups $\Delta_0$ or $\Delta_{-1}$. 
By Table 4.1 in \cite{indp2} (the description of $\mu_V(G_0^\vee)$), 
this would not be the case if 
$\Aut_{O^{p'}(\calf)}(A)$ were strictly contained in $\autf(A)\cong 
W(\gg)$. Thus $\Aut_{O^{p'}(\calf)}(A)=\autf(A)$, and so 
$O^{p'}(\calf)=\calf$ by Lemma \ref{theta-F-A}(c).

\noindent\textbf{Case (b.2) } Assume $|S/A|\ge p^2$; equivalently, 
$p^2\mid|W(\gg)|$. Then either $p=3$ and $\gg=F_4$ or $E_n$, or $p=5$ and 
$\gg=E_8$. In each case, there is an element $z\in Z(S)$ whose 
centralizer $H=C_G(z)$ is as described in the Table \ref{tbl:H/z}. The 
normalizer $N_G(z)$ is described in \cite[Table 5.1]{LSS}, 
and the descriptions of $O^p(H)/\gen{z}$ and $H/O^p(H)$ are based on that. 
A comparison of orders shows that $p\nmid[G:H]$ in each case, 
and hence that we can choose $z\in Z(S)$.
	\begin{table}[ht]
	\renewcommand{\arraystretch}{1.3} \renewcommand{\arraycolsep}{3mm}
	\[ \begin{array}{cc|ccc}
	\gg & p & O^p(H)/\gen{z}=O^p(H/\gen{z}) & H/O^p(H) & \4H 
	\\\hline
	F_4 & 3 & \PSL_3(q)\times\PSL_3(q) & C_3 & 
	\SL_3(\4\F_q)\circ\SL_3(\4\F_q) \\
	E_6 & 3 & \PSL_3(q)\times\PSL_3(q)\times\PSL_3(q) & C_3\times C_3 & 
	\SL_3(\4\F_q)\circ\SL_3(\4\F_q)\circ\SL_3(\4\F_q) \\
	E_7 & 3 & E_6(q)\times C_{(q-1)/3^\ell} & C_{3^{\ell}} & 
	3E_6(\4\F_q)\circ\4\F_q^\times \\
	E_8 & 3 & E_6(q)\times\PSL_3(q) & C_3 & 
	3E_6(\4\F_q)\circ\SL_3(\4\F_q) \\
	E_8 & 5 & \PSL_5(q)\times\PSL_5(q) & C_5 & 
	\SL_5(\4\F_q)\circ\SL_5(\4\F_q) \\
	\end{array} \]
	\caption{\small{In all cases, $G=\gg(q)$ where $\ell=v_p(q-1)\ge1$, 
	and $\4H=C_{\gg(\4{\F}_q)}(z)$ and $H=C_G(z)$ for a certain 
	$z\in Z(S)$ of order $p$. Also, $E_6(-)$ means the simple group and 
	$3E_6(-)$ its universal form, and ``$X\circ Y$'' denotes a central 
	product of $X$ and $Y$ over $C_p$.}}
	\label{tbl:H/z}
	\end{table} 

The description of the centralizer of $z$ in the algebraic group 
$\gg(\4\F_q)$ in the last column of the table (taken from \cite[Table 
VI]{Griess}), is included to better explain the structure of $H$. For 
example, when $p=3$ and $G=F_4(q)$, $H$ is an extension of 
$\SL_3(q)\circ\SL_3(q)$ by $C_3$ acting via diagonal automorphisms of order 
$3$ on each factor. 

Set $\cale=C_\calf(z)=\calf_S(H)$. If $\ell=v_p(q-1)\ge2$ or $p=5$, then 
the $p$-fusion system of $\PSL_p(q)$ is simple by Proposition 
\ref{p:p'-ind-1}(a), and $O^{p'}(O^p(\cale)/\gen{z})$ is the fusion system 
of $O^{p'}(O^p(H)/\gen{z})$ by the table. By Lemmas \ref{l:Op'Op<Op'} 
and \ref{mod-Z(F)}, $O^{p'}(\cale)=\cale$, except when $(\gg,p)=(E_7,3)$, in 
which case $O^{p'}(\cale)=\calf_S(O^{p'}(H))$. So by \eqref{e:AutF(A)/Op'}, 
$\autf(A)$ is generated by $O^{p'}(\autf(A))$ and $\Aut_{O^{p'}(\cale)}(A)$, 
and $O^{p'}(\calf)=\calf$ by Lemma \ref{l:Op'F=F:2}. 

Now assume $\ell=1$ and $p=3$. Set $K=\PSL_3(q)$ and fix $U\in\syl3K$. Then 
$U\cong C_3\times C_3$, so $\calf_U(K)=\calf_U(N_K(U))\cong\calf_U(U\rtimes Q_8)$. If 
$\gg=F_4$, then $\cale/\gen{z}$ is the fusion system of an extension of 
$(C_3)^4\rtimes(Q_8\times Q_8)$ by $C_3$ acting faithfully on each factor 
$Q_8$, so $O^{3'}(\cale/\gen{z})=\cale/\gen{z}$, and $O^{3'}(\cale)=\cale$ 
by Lemma \ref{mod-Z(F)}. A similar argument shows that 
$O^{3'}(\cale)=\cale$ when $\gg=E_6$. This, together with the descriptions 
in Table \ref{tbl:H/z}, show that $\autf(A)$ is generated by 
$O^{p'}(\autf(A))$ and $\Aut_{O^{p'}(\cale)}(A)$ in all four cases (also 
when $\gg=E_7$ or $E_8$) and hence that $O^{p'}(\calf)=\calf$ by Lemma 
\ref{l:Op'F=F:2} again.

\smallskip

\noindent\textbf{(c) } Now assume $p=5$, and $G=E_8(q)$ where $q\equiv\pm2$ 
(mod $5$). Set $\ell=v_5(q^4-1)=v_5(q^2+1)$. By \cite[Lemma 6.7]{BMO2}, 
there is a unique $A=(C_{5^\ell})^4$ of index $5$ in $S$, and 
$\Aut_\calf(A)\cong(C_4\circ2^{1+4}).\Sigma_6$ (where ``$\circ$'' denotes 
the central product over $C_2$). We are thus in the situation of 
\cite[Section 2]{indp2}. Since $\rk(A)\le p$, we have $|Z(S)|=p$, and hence 
$Z=Z_0$ and $\autf^\vee(S)=\autf(S)$ in the notation of 
\eqref{e:COS-defs}. So $O^{p'}(\calf)=\calf$ by Table 2.2 and Lemma 
2.7(c.i) in \cite{indp2}.

\smallskip

\noindent\textbf{(d) } This case, where $G=G_2(q)$ and $p=3$, was handled 
by Aschbacher \cite[16.11]{A-gfit}. In particular, he showed there that 
$S\nsg\calf$ if $q\equiv\pm4,\pm5$ (mod $9$), while $O_3(\calf)=Z(S)$ and 
$O^{3'}(\calf)=C_\calf(Z(S))$ has index $2$ in $\calf$ and is realized by 
$\SL_3^\pm(q)$ if $q\equiv\pm1$ (mod $9$). 
\end{proof}

With the help of Propositions \ref{p:p'-ind-1} and \ref{p:p'-ind-2} and 
Tables \ref{tbl:m-mu-theta} and \ref{tbl:AutG(A)-except}, we can now 
check that certain fusion systems are exotic.

\begin{Lem} \label{l:exotic}
Assume the classification of finite simple groups. Let $\calf$ be a 
saturated fusion system over a finite $p$-group $S$, for some prime 
$p\ge5$, and assume $A\nsg S$ is abelian and $\calf$-centric. 
Assume also, for some $\ell\ge1$, $\kappa\ge p$, and 
$2< m\mid(p-1)$, that $A$ is homocyclic of rank $\kappa$ and exponent 
$p^\ell$, and that with respect to some basis $\{a_1,\dots,a_\kappa\}$ 
for $A$ as a $\Z/p^\ell$-module, $\autf(A)$ contains $G(m,m,\kappa)$ with 
index prime to $p$, and 
	\[ \autf(A)\cap G(m,1,\kappa)=G(m,r,\kappa) \le 
	\GL_\kappa(\Z/p^\ell) \qquad\textup{for some $2<r\mid m$.} \]
Then either $A\nsg\calf$ or $\calf$ is exotic. 
\end{Lem}

\begin{proof} Assume otherwise: assume $A\nnsg\calf$ and $\calf$ is 
realized by the finite group $G$. Assume also that $|G|$ is minimal among 
all such counterexamples $(A,\calf,G)$ to the lemma. In particular, 
$|G|$ is minimal among all orders of finite groups realizing $\calf$, and 
so $O_{p'}(G)=1$. 

If $\ell>1$ and $\Omega_1(A)\nsg\calf$, then $\Omega_1(A)\nsg G$ by the 
minimality assumption ($\calf$ is realized by $N_G(\Omega_1(A))$). So in 
this case, $A/\Omega_1(A)\nnsg\calf/\Omega_1(A)$ where $\calf/\Omega_1(A)$ 
is realized by $G/\Omega_1(A)$, which again contradicts our minimality 
assumption on $|G|$. Thus $\Omega_1(A)\nnsg\calf$, and hence 
$\Omega_k(A)\nnsg\calf$ for each $1\le k\le\ell$. 

Since $\autf(A)$ contains $G(m,r,\kappa)$ with index prime to $p$, 
some subgroup conjugate to $\Aut_S(A)$ is contained in 
$G(1,1,\kappa)\cong\Sigma_\kappa$, and hence $\Aut_S(A)$ permutes some 
basis of $\Omega_1(A)$. Also, $G(m,r,\kappa)$ acts faithfully on 
$\Omega_1(A)$, as does each subgroup of $\autf(A)$ of order prime to $p$ 
(see \cite[Theorem 5.2.4]{Gorenstein}). So by the assumptions on 
$\autf(A)$, 
	\beqq \parbox{95mm}{$\autf(A)$ acts faithfully on $\Omega_1(A)$, 
	$\Aut_S(A)$ permutes a basis of $\Omega_1(A)$, and 
	$C_S(\Omega_1(A))=A$.} \label{e:CS(O1(A))}
	\eeqq

We claim that
	\beqq \textup{$\Omega_1(A)$ is the only elementary abelian subgroup 
	of $S$ of rank $\kappa$.} \label{e:A-uniq} \eeqq
This is well known, but the proof is simple enough that we give it here. Set 
$V=\Omega_1(A)$ for short, let $W\le S$ be another elementary abelian 
subgroup, and set $\4W=\Aut_W(V)$ and $r=\rk(\4W)$. Then $r=\rk(W/(V\cap 
W))$ since $C_W(V)=W\cap A=W\cap V$ by \eqref{e:CS(O1(A))}. Let $\calb$ be 
a basis for $V$ permuted by $\4W$, and assume $\4W$ acts on 
$\calb$ with $s$ 
orbits (including fixed orbits) of lengths $p^{m_1},\dots,p^{m_s}$. Then 
$m_1+\dots+m_s\ge r=\rk(\4W)$, and hence 
	\[ \rk(W) = \rk(\4W)+\rk(W\cap V) \le r+\rk(C_V(\4W)) 
	= r + s \le \sum_{i=1}^s(1+m_i) < 
	\sum_{i=1}^sp^{m_i}=\rk(V), \]
proving \eqref{e:A-uniq}. In particular, $\Omega_1(A)$ and 
$A=C_S(\Omega_1(A))$ are weakly closed in $\calf$.

Set $G_0=O^{p'}(G)$ and $\calf_0=\calf_S(G_0)\nsg\calf$, and let 
$\Gamma=\autf(A)$ and $\Gamma_0=\Aut_{\calf_0}(A)$ for short. By Lemma 
\ref{l:QnsgEnsgF} and since $\Omega_1(A)$ is weakly closed in $\calf$, we 
have $\Omega_1(A)\nnsg\calf_0$.

\smallskip

\noindent\textbf{Step 1: } We first claim that $G_0$ is simple. Assume 
otherwise: let $H\nsg G_0$ be a proper nontrivial normal subgroup, and set 
$T=H\cap S\in\sylp{H}$. Thus $1\ne T<S$ since $O_{p'}(G_0)=1$ and 
$O^{p'}(G_0)=G_0$, and $T$ is strongly closed in $\calf_0$. 

Now, $\Aut_H(A)\nsg\Gamma_0$ since $H\nsg G_0$, and by assumption, 
$\Gamma_0$ contains $G(m,m,\kappa)$ with index prime to $p$. So 
$\Aut_H(A)\cap G(m,m,\kappa)$ is normal in 
$G(m,m,\kappa)\cong(C_m)^{\kappa-1}\rtimes\Sigma_\kappa$, where $\kappa\ge 
p\ge5$. Consequently, either $\Aut_H(A)\cap G(m,m,\kappa)$ and hence $\Aut_H(A)$ 
have order prime to $p$, or $\Aut_H(A)$ contains 
$O^{p'}(G(m,m,\kappa))$. In the latter case, 
$H\ge[N_H(A),A]\ge[O^{p'}(G(m,m,\kappa)),A]=A$, and hence $N_H(A)$ has 
index prime to $p$ in $N_G(A)$. But then $T=S$, a contradiction. 

Thus $p\nmid|\Aut_H(A)|$, and so $T=O^{p'}(N_H(A))\le A$. Also, $T$ 
is normalized by $\Gamma_0\ge O^{p'}(G(m,m,\kappa))$, and so 
$T=\Omega_k(A)$ for some $1\le k\le\ell$. But then $T$ is abelian and 
strongly closed in $\calf$, so $T\nsg\calf_0$ by \cite[Corollary 
I.4.7(a)]{AKO}, hence $\Omega_1(A)=\Omega_1(T)\nsg\calf_0$, which we 
already showed is impossible. So there is no such $T$, and $G_0$ 
must be simple.

\smallskip

\noindent\textbf{Step 2: } It remains to show that $G_0$ cannot be any 
known simple group. Note that $A$ is a radical $p$-subgroup of $G_0$, since 
$O_p(\Aut_{G_0}(A))=1$ and $p\nmid|C_{G_0}(A)/A|$ (i.e., $A$ is 
$\calf_0$-centric). Although we do not know $\Aut_{G_0}(A)$ precisely, we 
know that it is contained in $\autf(A)$ and contains 
$O^{p'}(G(m,m,\kappa))\cong(C_m)^{\kappa-1}\rtimes A_\kappa$.

Since $p\ge5$ and $\rk_p(G_0)\ge p$, $G_0$ cannot be a sporadic group by 
\cite[Table 5.6.1]{GLS3}. 

By \cite[\S\,2]{AF}, for each abelian radical $p$-subgroup 
$B\le\Sigma_\kappa$, $\Aut_{\Sigma_\kappa}(B)$ is a product of wreath 
products of the form $\GL_c(p)\wr\Sigma_\kappa$ for $c\ge1$ and 
$\kappa\ge1$. Thus $\Aut_{A_\kappa}(B)$ can have index $2$ in 
$C_{p-1}\wr\Sigma_\kappa$ for some $\kappa$, but not index larger than $2$. 
So $G_0$ cannot be an alternating group.

If $G_0\in\Lie(p)$, then $N_{G_0}(A)$ is a parabolic subgroup by the 
Borel-Tits theorem \cite[Corollary 3.1.5]{GLS3} and since $A$ is centric 
and radical. So in the notation of \cite[\S\,2.6]{GLS3}, $A=U_J$ and 
$N_{G_0}(A)=P_J$ (up to conjugacy) for some set $J$ of primitive roots for 
$G_0$. Hence by \cite[Theorem 2.6.5(f,g)]{GLS3}, 
$O^{p'}(N_{G_0}(A)/A)\cong O^{p'}(L_J)$ is a central product of groups in 
$\Lie(p)$, contradicting the assumption that $O^{p'}(\Aut_G(A))\cong 
G(m,r,\kappa)$. 

Now assume that $G_0\in\Lie(q_0)$ for some prime $q_0\ne p$. By 
\cite[10-2]{GL} (and since $p\ge5$), $S\in\sylp{G_0}$ contains a unique 
elementary abelian $p$-subgroup of maximal rank, and by 
\eqref{e:A-uniq}, it must be equal to $\Omega_1(A)$. Hence 
$\autf(A)$ must be as in one of the entries in Table 
\ref{tbl:m-mu-theta} or \ref{tbl:AutG(A)-except}. \\
$\bullet$ If $G_0$ is a classical group and hence 
$\autf(A)\cong G(\5m,\5r,\kappa)$ for $\5m=\mu$ or $2\mu$ and $\5r\le2$ 
(see the next-to-last column in Table \ref{tbl:m-mu-theta}, 
and recall that $G(\5m,1,\kappa)\cong C_{\5m}\wr\Sigma_\kappa$), 
then the identifications $\autf(A)\cong G(\5m,\5r,\kappa)$ and $\autf(A)\ge 
G(m,r,\kappa)$ are based on the same decompositions of $A$ as a direct sum 
of cyclic subgroups, and hence we have $m\mid\5m$ and $r\le2$, 
contradicting our original assumption. \\
$\bullet$ If $G_0$ is an 
exceptional group, then by Table \ref{tbl:AutG(A)-except}, either 
$\kappa=\rk(A)<p$, or $p=3$, or (in case (b)) $m^{\kappa-1}\cdot\kappa!$ 
does not divide $|\autf(A)|$ for any $m>2$ and hence $\autf(A)$ cannot 
contain any such $G(m,r,\kappa)$. 
\end{proof}

The results in this section are now summarized in the following theorem, 
in which the statements that certain fusion systems are exotic depend 
on the classification of finite simple groups. 
Recall that a fusion subsystem of $\calf$ is \emph{characteristic} in 
$\calf$ if it gets sent to itself by each $\alpha\in\Aut(\calf)$.

\begin{Thm} \label{F.not.simple}
Fix a prime $p$ and a known finite simple group $G$ such that $p\mid|G|$. 
Fix $S\in\sylp{G}$, and set $\calf=\calf_S(G)$. Then either 
\begin{enuma} 

\item $S\nsg\calf$; or 

\item $p=3$ and $G\cong G_2(q)$ for some $q\equiv\pm1$ (mod $9$), 
in which case $|O_3(\calf)|=3$, $|\quotfus[3]\calf|=2$, and 
$O^{3'}(\calf)$ is realized by $\SL_3^\pm(q)$; or 

\item $p\ge5$ and $G$ is one of the classical groups $\PSL_n^\pm(q)$, 
$\PSp_{2n}(q)$, $\varOmega_{2n+1}(q)$, or $\POmega_{2n+2}^\pm(q)$ where $n\ge2$ 
and $q\not\equiv0,\pm1$ (mod $p$), in which case $\quotfus\calf$ is cyclic 
and $O^{p'}(\calf)$ is simple and exotic; or 

\item $|\quotfus\calf|\le2$, $O^{p'}(\calf)$ is simple, and it is realized 
by a known finite simple group $G^*$. 

\end{enuma}
Moreover, the following hold in case \textup{(c)}. 
\begin{enumerate}[\rm({c.}1) ]

\item There are a subsystem $\calf_0\nsg\calf$ of index at most 
$2$ in $\calf$, and a finite group $G_0$ realizing $\calf_0$, with the 
following properties. For each saturated fusion system $\cale$ over $S$ 
such that $O^{p'}(\cale)=O^{p'}(\calf)$, either 
\begin{itemiz} 
\item $\cale\ge\calf_0$, and $\cale=\calf_S(H)$ for some group $H$ with 
$G_0\nsg H$ and $p\nmid|H/G_0|$; or 
\item $\cale\ngeq\calf_0$, and $\cale$ is exotic.
\end{itemiz}

\item If $\cale\ge\calf$ is an extension of index prime to $p$ (i.e., 
$\cale$ is saturated and $\calf$ is normal of index prime to $p$ in 
$\cale$), then $\calf$ is characteristic in $\cale$.

\end{enumerate}
\end{Thm}

\begin{proof} Assume throughout the proof that point (a) does not hold; 
i.e., that $S\nnsg\calf$. By Propositions \ref{t:tame.simple}, 
\ref{p:2'-index}, and \ref{p:p'-ind-2}, point (b) or (d) holds except 
possibly when $p$ is odd and $G$ is one of the classical groups 
$\PSL_n^\pm(q)$, $\PSp_{2n}(q)$, $\POmega_{2n+1}(q)$, or 
$\POmega_{2n}^\pm(q)$ for $p\nmid q$. So assume we are in one of those 
cases. In particular, $O^{p'}(\calf)$ is simple by Proposition 
\ref{p:p'-ind-1}.

If $G\cong\PSU_n(q)$ for some $n$ and $q$, then by Lemma 
\ref{equiv-simp-2}(a), $G\sim_p\PSL_n(q^\vee)$ for some prime power 
$q^\vee$ with $\ord_p(q^\vee)=\ord_p(-q)$. So we can assume $G$ is 
linear, symplectic, or orthogonal.

\noindent\textbf{Case 1: } Assume $q\equiv\pm1$ (mod $p$). If $G=\PSL_n(q)$ 
and $q\equiv1$ (mod $p$), then $O^{p'}(\calf)=\calf$ by Proposition 
\ref{p:p'-ind-1}(a). If $G=\PSL_n(q)$ and $q\equiv-1$ (mod $p$), so that 
$\ord_p(q)=2$, then set $k=[n/2]$. We have 
	\[ G \sim_p \GL_n(q) \sim_p \GL_{2k}(q) \sim_p \GO_{2k}^\gee(q) \]
for $\gee=(-1)^k$ by Lemma \ref{equiv-simp-1}(a,d). Also, $O^{p'}(\calf)$ 
has index $2$ in $\calf$ by Proposition \ref{p:p'-ind-1}(a,d), and it is 
realized by the simple group $\POmega_{2n}^\gee(q)$ by Table 
\ref{tbl:m-mu-theta} and Lemma \ref{theta-F-A}(c). 

If $G$ is one of the groups $\PSp_{2n}(q)$ or $\varOmega_{2n+1}(q)$ for 
$q\equiv\pm1$ (mod $p$), or $\POmega_{2n}^\gee(q)$ for $q\equiv\pm1$ and 
$q^n\not\equiv\gee$ (mod $p$), then a similar argument shows that 
$O^{p'}(\calf)$ has index $2$ in $\calf$ and is realizable by a simple 
orthogonal group. If $G=\POmega_{2n}^\gee(q)$ where $q\equiv\pm1$ and 
$q^n\equiv\gee$ (mod $p$), then $O^{p'}(\calf)=\calf$ by Proposition 
\ref{p:p'-ind-1}(c). 

Thus $O^{p'}(\calf)$ is realizable (and we are in the situation of 
(c)) whenever $G$ is one of the classical groups defined over $\F_q$ 
and $q\equiv\pm1$ (mod $p$). 

\noindent\textbf{Case 2: } Assume $q\not\equiv0,\pm1$ (mod $p$), where $G$ 
is again one of the classical groups $\PSL_n(q)$, $\PSp_{2n}(q)$, 
$\POmega_{2n+1}(q)$, or $\POmega_{2n}^\pm(q)$. In particular, $p\ge5$. Set 
$m=\ord_p(q)$ and $\ell=v_p(q^m-1)$; thus $\ell\ge1$ and $2<m\mid(p-1)$. By 
Table \ref{tbl:m-mu-theta}, there are $\kappa\ge1$ and $A\nsg S$ weakly 
closed in $\calf$ such that $A\cong(C_{p^\ell})^\kappa$ and 
$\Aut_{O^{p'}(\calf)}(A)\cong G(m,m,\kappa)$. So $O^{p'}(\calf)$ is exotic 
by Lemma \ref{l:exotic}.

\smallskip

\noindent\textbf{Proof of (c.1): } Assume $\cale$ is a saturated fusion system over 
$S$ such that $O^{p'}(\cale)=O^{p'}(\calf)$. Then 
$\Aut_{O^{p'}(\calf)}(A)\nsg\Aut_{\cale}(A)$. Let $\{a_1,\dots,a_\kappa\}$ 
be a basis for $A$ as a $\Z/p^\ell$-module such that with respect to this 
basis, $\Aut_{O^{p'}(\calf)}(A)=G(m,m,\kappa)$: the group of monomial 
matrices generated by the permutation matrices, and the diagonal matrices 
of order dividing $m$ and of determinant $1$. 

Now, $\Aut_{\cale}(A)\le N_{\Aut(A)}(\Aut_{O^{p'}(\calf)}(A))= 
G(m,1,\kappa)\gen{(\Z/p^\ell)^\times\cdot\Id_A}$: 
each $\alpha\in\Aut(A)$ that normalizes $G(m,m,\kappa)$ permutes the 
subgroups $\gen{a_i}$ and thus has a monomial matrix whose nonzero 
entries differ pairwise by $m$-th roots of unity. Furthermore, since 
$\Aut_{O^{p'}(\calf)}(A)$ has index prime to $p$ in $\Aut_\cale(A)$, we 
have 
	\beqq \Aut_{\cale}(A) \le G(m,1,\kappa) \gen{\omega\cdot\Id_A}, 
	\label{e:G<w>} \eeqq
for $\omega\in\Z$ whose class in $(\Z/p^{\ell+1})^\times$ has order exactly 
$p(p-1)$.

By Dirichlet's theorem on primes in arithmetic sequences, we can 
choose a prime $q_0\equiv\omega$ (mod $p^{\ell+1}$). In particular, 
$v_p(q_0^{p-1}-1)=\ell$ and the class of $q^\vee$ generates 
$(\Z/p)^\times$. Set $q^*=q_0^a$, where $a$ is chosen so that $p\nmid a$ 
and $q^*\equiv q$ (mod $p$). Upon replacing $q$ by $q^*$ and $G=\gg^\pm(q)$ 
by $\gg^\pm(q^*)$, we can arrange that $q=q_0^a$ without changing the 
fusion system $\calf$ (see \cite[Theorem A(a)]{BMO1}).

Let $G_0\nsg G_1$ be as defined in Table \ref{tbl:G0-G1}, and set 
$\calf_0=\calf_S(G_0)$. In the first two cases shown in the table, the 
isomorphism in the third column follows from Table \ref{tbl:m-mu-theta}, 
and $G_0=G$ and $\calf_0=\calf$. In the third 
case, $G\sim_p\PGO_{2n'}^\gee(q)$ by Lemma \ref{equiv-simp-1} (and Table 
\ref{tbl:m-mu-theta}), and we identify $S$ as a subgroup of 
$G_0=\POmega_{2n'}^\gee(q)$ in such a way that $\calf_0\le\calf$. Then 
$\Aut_{\calf_0}(A)=G(m,2,\kappa)$ and hence $\calf_0$ has index $2$ in 
$\calf$. In all cases, $\psi_{q_0}$ denotes the field automorphism 
$(x\mapsto x^{q_0})$, this acts on $A$ via $(a\mapsto 
a^{q_0}=a^\omega)$, and so 
$\Aut_{G_1}(A)=G(m,1,\kappa)\gen{\omega\cdot\Id}$. 
	\begin{table}[ht] 
	\[ \renewcommand{\arraystretch}{1.2} \setlength{\arraycolsep}{4mm}
	\begin{array}{cc|ccc}
	m & \autf(A) & G & G_0 & G_1 \\\hline
	\textup{odd} & G(m,1,\kappa) & G\cong\PSL_n(q) & \PSL_n(q) & 
	\PSL_n(q)\gen{\psi_{q_0}} \\
	\textup{even} & G(m,2,\kappa) & G\cong\POmega_{2n'}^\gee(q) & 
	\POmega_{2n'}^\gee(q) & \PGO_{2n'}^\gee(q)\gen{\psi_{q_0}} \\
	\textup{even} & G(m,1,\kappa) & G\sim_p\PGO_{2n'}^\gee(q) & 
	\POmega_{2n'}^\gee(q) & \PGO_{2n'}^\gee(q)\gen{\psi_{q_0}} 
	\end{array} \]
	\caption{\small{In the last two cases, $n'\ge p$ and $\gee=\pm1$ are 
	such that $q^{n'}\equiv\gee$ (mod $p$).}}
	\label{tbl:G0-G1}
	\end{table}

Thus in all three cases, 
	\beqq G(m,m,\kappa) = \Aut_{O^{p'}(\calf)}(A) \le \Aut_\cale(A)
	\le G(m,1,\kappa)\gen{\omega\cdot\Id} = \Aut_{G_1}(A), 
	\label{e:G<Aut<G} \eeqq
where the second inequality follows from \eqref{e:G<w>}. Hence 
	\beqq \Aut_\cale(A)\cap G(m,1,\kappa) = G(m,r,\kappa) \qquad 
	\textup{for some $r\mid m$.} \label{e:G(m,r,k)} \eeqq
If $r>2$, then $\cale$ is exotic by Lemma \ref{l:exotic}, and 
$\cale\ngeq\calf_0$ since $\Aut_{\calf_0}(A)=G(m,\gcd(m,2),\kappa)$. 

Now assume $r\le2$, and hence that $\Aut_\cale(A)\ge\Aut_{\calf_0}(A)$. By 
\eqref{e:G<Aut<G}, there is $H\le G_1$ such that $G_0\nsg H$ and 
$\Aut_H(A)=\Aut_\cale(A)$. Since 
$O^{p'}(\calf_S(H))=O^{p'}(\calf_0)=O^{p'}(\cale)$, we have 
$\cale\cong\calf_S(H)$ by Lemma \ref{Op'F1=Op'F2}. 

We claim that this is, in fact, an equality (and hence that 
$\cale\ge\calf_0$). 
Set $\Gamma=\Aut_{O^{p'}(\calf)}(A)=G(m,m,\kappa)$ and 
$\Gamma_0=O_{p'}(\Gamma)\cong(C_m)^{\kappa-1}$ for short.
The five-term exact sequence for the cohomology of $\Gamma$ as an extension 
of $\Gamma_0$ by $\Gamma/\Gamma_0$ (see \cite[p. 110]{Benson2}) 
restricts to an exact sequence 
	\beqq 0 \too 
	\underset{=0}{H^1(\Gamma/\Gamma_0;H^0(\Gamma_0;A))} 
	\Right3{} H^1(\Gamma;A) \Right3{} 
	\underset{=0}{H^0(\Gamma/\Gamma_0;H^1(\Gamma_0;A))} 
	\label{e:H1=0} \eeqq
in which the first term is zero since 
$H^0(\Gamma_0;A)=C_A(\Gamma_0)=1$, and the last term is zero 
since $\gcd(|\Gamma_0|,|A|)=1$. Thus $H^1(\Gamma;A)=0$, and so 
$\cale=\calf_S(H)$ by Lemma \ref{Op'F1=Op'F2} again. In particular, 
$\cale\ge\calf_0$ in this case, and this finishes the proof of (a.1).

\smallskip

\noindent\textbf{Proof of (c.2): } Now assume $\cale\ge\calf$ is an 
extension of index prime to $p$, and let $\alpha$ be 
an automorphism of $\cale$. Thus $\alpha\in\Aut(S)$ is such that 
$\9\alpha\cale=\cale$, and in particular, 
$\9\alpha(O^{p'}(\calf))=O^{p'}(\calf)$, $\alpha(A)=A$, and 
$\9\alpha(\Aut_{\cale}(A))=\Aut_{\cale}(A)$. Since by \eqref{e:G<w>}, $\autf(A)$ is 
generated by $\Aut_{O^{p'}(\calf)}(A)$ together with either the $m$-torsion 
in $O_{p'}(\Aut_{\cale}(A))$ or the unique normal subgroup of index $2$ 
in that subgroup, we also have $\9\alpha(\autf(A))=\autf(A)$. Since 
$H^1(\Aut_{O^{p'}(\calf)}(A);A)=0$ by \eqref{e:H1=0}, we have 
$\9\alpha\calf=\calf$ by Lemma \ref{Op'F1=Op'F2}. Since 
$\alpha\in\Aut(\cale)$ was arbitary, $\calf$ is characteristic in 
$\cale$. 
\end{proof}




\end{document}